\documentclass[11pt,reqno]{amsart}
\usepackage{amsmath,amssymb,graphicx,mathrsfs,amsopn,stmaryrd,amsthm,amscd,color,bm,extarrows}
\usepackage{newtxtext}
\usepackage{ytableau,textcase}
\usepackage[colorlinks=true,citecolor=blue,linkcolor=blue
]{hyperref}
\usepackage[alphabetic,nobysame]{amsrefs}
\renewcommand{\MR}[1]{} \renewcommand{\PrintDOI}[1]{}
\newcommand{\arxiv}[1]{\href{http://arxiv.org/abs/#1}{arXiv:\nolinkurl{#1}}}
\usepackage[english]{babel}
\usepackage[mathscr]{euscript}
\usepackage[left=2.5cm,
            right=2.5cm,
            top=2.75cm,
            bottom=2.75cm
            ]{geometry}
\usepackage{collectbox}


\usepackage{dsfont}
\usepackage{tikz}    
\usetikzlibrary{arrows,matrix,decorations.markings,shapes.geometric}   
\usetikzlibrary{decorations.pathreplacing}

\usepackage{xy}
\allowdisplaybreaks
\xyoption{all}
\numberwithin{equation}{section}
\newtheorem{thm}{Theorem}[section]
\newtheorem{prop}[thm]{Proposition}
\newtheorem{lem}[thm]{Lemma}
\newtheorem{cor}[thm]{Corollary}
\newtheorem{conj}[thm]{Conjecture}

\newtheorem{alphatheorem}{Theorem}

\theoremstyle{definition} 

\newtheorem{dfn}[thm]{Definition}
\theoremstyle{remark}
\newtheorem{rem}[thm]{Remark}

\newcommand{\beq}{\begin{equation}}
\newcommand{\eeq}{\end{equation}}
\newcommand{\be}{\begin{equation*}}
\newcommand{\ee}{\end{equation*}}

\newcommand{\bC}{\mathbb{C}}

\newcommand{\bZ}{\mathbb{Z}}

\newcommand{\bN}{\mathbb{N}}

\newcommand{\mc}{\mathcal}

\newcommand{\sfh}{\mathsf{h}}
\newcommand{\sfb}{\mathsf{b}}

\newcommand{\sfs}{\mathsf{s}}

\newcommand{\hf}{\tfrac12}

\newcommand{\bb}{\mathbf{b}}
\newcommand{\bh}{\mathbf{h}}

\newcommand{\g}{\mathfrak{g}}

\newcommand{\gl}{\mathfrak{gl}}

\newcommand{\fksl}{\mathfrak{sl}}
\newcommand{\fkS}{\mathfrak{S}}

\newcommand{\gr}{{\mathrm{gr}}}

\newcommand{\tl}{\tilde}

\newcommand{\gge}{\geqslant}
\newcommand{\lle}{\leqslant}
\newcommand{\la}{\lambda}

\newcommand{\bla}{\bm\lambda}

\newcommand{\I}{{\mathbb I}}

\newcommand{\Y}{{\mathscr{Y}}}
\newcommand{\Yi}{{^\imath\mathscr{Y}}}

\newcommand{\YiJ}{{^\imath\mathscr{Y}_{\jmath}}}

\newenvironment{nouppercase}{%
  \renewcommand{\uppercasenonmath}[1]{}}{}

\begin{document}
\pagestyle{myheadings}
\setcounter{page}{1}

\title[Minimal presentation and coideal structure of twisted Yangians]{\Large Minimalistic presentation and coideal structure of twisted Yangians}

\author{Kang Lu}
\date{\today}
\address{
Shenzhen International Center for Mathematics and Department of Mathematics, Southern University of Science and Technology, Shenzhen, China\footnote{The author was previously at the University of Virginia when this research was initiated.}
}
\email{kanglu.math@outlook.com,luk@sustech.edu.cn}

\dedicatory{Dedicated to the memory of Chen-Ning Yang}   


\begin{abstract}
We introduce a minimalistic presentation for the twisted Yangians ${}^\imath\mathscr Y$ associated with split symmetric pairs (or Satake diagrams) introduced in \cite{LWZ25affine} via a Drinfeld type presentation. As applications, we establish an injective algebra homomorphism from ${}^\imath\mathscr Y$ to the Yangian $\mathscr Y$, thereby identifying ${}^\imath\mathscr Y$ as a right coideal subalgebra of $\mathscr Y$ and proving its isomorphism with the twisted Yangian in the $J$ presentation. Furthermore, we provide estimates for the Drinfeld generators of ${}^\imath\mathscr Y$ and describe their images under the coproduct in terms of the Drinfeld generators of $\mathscr Y$ under this identification.
\end{abstract}
\begin{nouppercase}	
\maketitle
\end{nouppercase}
\setcounter{tocdepth}{1}
\tableofcontents

\thispagestyle{empty}
\section{Introduction} 
\subsection{Background}
Yangians for $\gl_N$ first appeared in mathematical physics in the work of the St. Petersburg school in the late 1970s and early 1980s concerning the quantum inverse scattering method. Yangians associated with simple Lie algebras were introduced by Drinfeld \cite{Dri85} in the 1980s as a new class of quantum groups, arising naturally in the study of the Yang–Baxter equation and quantum integrable systems. The term ``Yangian" was introduced by Drinfeld in honor of C.N. Yang, who found the first nontrivial solution to Yang–Baxter equation. Yangians provide deep algebraic insights into representation theory, quantum integrable models, and symmetries in field theory, while their connections with geometry and combinatorics continue to inspire developments across several areas of mathematics.

Twisted Yangians are a family of important quantum algebras closely related to Yangians; they correspond to symmetric pairs $(\g,\g^{\theta})$, where $\mathfrak{g}$ is a finite-dimensional simple complex Lie algebra and $\mathfrak{g}^{\theta}$ is the fixed-point subalgebra of $\mathfrak{g}$ under an involution $\theta$. Twisted Yangians were first introduced by Olshanski \cite{Ols92} for symmetric pairs of types $\mathsf{AI}$ and $\mathsf{AII}$ via the R-matrix presentation. These twisted Yangians are closely related to classical Lie algebras of types $\mathsf{BCD}$ and their representations \cite{MNO96,Mo07}. Later, the R-matrix construction of twisted Yangians was extended to symmetric pairs of type $\mathsf{AIII}$ \cite{MR02} and to the remaining symmetric pairs of classical type \cite{GR16}. These algebras admit two complementary descriptions: on one hand, as abstract algebras governed by the reflection equation \cite{Che84,Skl88} together with additional constraints such as symmetry and unitary relations, and on the other hand, as coideal subalgebras inside the corresponding Yangians. Additionally, another family of twisted Yangians associated to arbitrary symmetric pairs was uniformly constructed in terms of Drinfeld's $J$ presentation \cite{Mac02}, arising as boundary remnants of Yangians in $1{+}1$-dimensional integrable field theories. Moreover, these algebras provide a homogeneous quantization of a Lie coideal structure for twisted current algebras of simple Lie algebras \cite{BR17}. These two family of twisted Yangians play a central role in quantum integrable systems with boundaries, integrable field theory, and AdS/CFT correspondence, see \cite{Skl88,Mac05,Reg24}, and they also serve as key ingredients in the study of fixed-point loci of quiver varieties \cite{Li19,Nak25} and affine Grassmannian slices \cite{LWW25}.

Recently, a Drinfeld type presentation of the twisted Yangian of type $\sf AI$ (split type $\sf A$) has been constructed in
\cite{LWZ25GD} by applying the Gauss decomposition for the generating matrices of twisted Yangians in the R-matrix presentation; such Drinfeld type presentations have been extended to twisted Yangians of
all split types \cite{LWZ25affine} and then to all quasi-split types \cite{LZ24}, by means of a degeneration procedure applied to the corresponding Drinfeld presentations of (quasi-)split affine $\imath$quantum groups \cite{LW21,mLWZ24,Zha22}.
However, an explicit identification of the  twisted Yangians in the R-matrix and  Drinfeld presentations has been established only for (quasi-)split type $\sf A$.

The Drinfeld presentation of twisted Yangians has proved to be a powerful tool, with applications to diverse areas including the study of Slodowy slices in classical Lie algebras \cite{TT24}, finite $\mathcal W$-algebras of classical types \cite{LPTTW25}, the geometry of the affine Grassmannian \cite{LWW25} (see also \cite{BPT25} for type $\mathsf{AI}$), and Coulomb branches of $3$-dimensional $\mathcal{N}=4$ gauge theories \cite{SSX25}.

\subsection{The problems}\label{sec:goal}
Let $\g$ be a finite-dimensional simple complex Lie algebra. Let $\Y$, $\Y_\jmath$, and $\Y_{\mathscr R}$ be the corresponding Yangians in Drinfeld, $J$, and R-matrix presentations, respectively. Let $\Yi$, $\YiJ$, and $\Yi_{\mathscr R}$ be the twisted Yangians associated to split type Satake diagrams (which coincide with Dynkin diagrams) in Drinfeld, $J$, and R-matrix presentations, respectively. These  diagrams correspond to symmetric pairs $(\g,\g^{\omega})$, where $\omega$ is the Chevalley involution.

The equivalence between $\Y$ and $\Y_\jmath$ is established in \cite{Dri87,GRW19} for all simple Lie algebras while isomorphisms between $\Y_\jmath$ and $\Y_{\mathscr R}$ are obtained in  \cite{Wen18}. Explicit isomorphisms between $\Y$ and $\Y_{\mathscr R}$ are established in \cite{BK05} for type $\mathsf A$ and in \cite{JLM18,GRW19} for types $\mathsf{BCD}$.

It is a natural and important question to establish equivalences between twisted Yangians in different presentations. The recipe to establish the isomorphism between $\YiJ$ and $\Yi_{\mathscr R}$ (classical type from \cite{Ols92,MR02,GR16}) is more or less clear, see \cite[Corollary 3.19]{CGM14} for the case of type $\mathsf{AIII}$ and cf. \cite[\S11]{Kolb14} for affine $\imath$quantum groups of types $\mathsf{AI}$ and $\mathsf{AII}$. Note that both $\YiJ$ and $\Yi_{\mathscr R}$ are subalgebras of $\Y$. One can show that $\YiJ$ is contained in $\Yi_{\mathscr R}$ by checking the generators of $\YiJ$ are in $\Yi_{\mathscr R}$, and hence they must be equal as their associated graded are equal. Alternatively, both $\YiJ$ and $\Yi_{\mathscr R}$ are coideal subalgebras of $\Y$ and homogeneous quantizations of a Lie coideal structure for the twisted current algebras, see \cite[\S3.5]{GR16}. Thus it follows from a uniqueness of homogeneous quantization result for symmetric pairs from \cite{BR17} that $\YiJ$ and $\Yi_{\mathscr R}$ are isomorphic.

In this paper, we focus on twisted Yangians $\Yi$ and $\YiJ$ in Drinfeld and $J$ presentations. The goal of this paper is to address the following:
\begin{enumerate}
    \item identify the twisted Yangian $\Yi$ in the Drinfeld presentation as a coideal subalgebra of the Yangian $\Y$,
    \item relate the twisted Yangians $\Yi$ and $\YiJ$ in the Drinfeld and $J$ presentations,
    \item approximately express Drinfeld generators of twisted Yangians in terms of Drinfeld generators of Yangians,
    \item and estimate the coproduct of Drinfeld generators of twisted Yangians in terms of Drinfeld generators of Yangians.
\end{enumerate}
Let us explain the main results of the paper in more detail.

\subsection{Minimalistic presentation}
To realize the twisted Yangian $\Yi$ in Drinfeld presentation as a subalgebra of the Yangian $\Y$, we construct an injective algebra homomorphism from $\Yi$ to $\Y$. Verifying that this map preserves all defining relations of $\Yi$ is typically highly nontrivial. This obstacle can, however, be resolved by employing a minimalistic presentation in the sense of \cite{Lev93gen,GNW18}.

Roughly speaking, a minimalistic presentation for Yangian $\Y$ can be obtained as follows. Let $\xi_{i,r}$ and $x_{i,r}^\pm$ for $i\in \I$ and $r\in\bN$ be the Drinfeld generators of $\Y$ with its defining relations, where $\I$ is an index set for the nodes of the Dynkin diagram of $\g$ and the subindex $r$ can be understood as the degree of the corresponding generators. It is well known that $\Y$ is generated by the degree-0 and degree-1 elements $\xi_{i,r}$ and $x_{i,r}^\pm$ for $i\in \I$ and $r=0,1$. Then one only takes a subset of defining relations involving only these elements as the defining relations for the minimalistic presentation. It turns out the algebra defined by the minimalistic presentation is isomorphic to the Yangian \cite{Lev93gen,GNW18}. Note that for the special rank 1 case, an extra relation $\big[\xi_{i,1},[x_{i,1}^+,x_{i,1}^-]\big]=0$ should also be included, see \cite[\S2]{GNW18} for discussions. This relation was included in \cite{Lev93gen} for all types but turns out to be redundant except for type $\mathsf A_1$.

Such a minimalistic presentation involves only finitely many generators and significantly fewer relations, making it easier to verify than the original definition and thereby highly useful for applications. For example, it has been employed to construct an explicit isomorphism between the Drinfeld and $J$ presentations of Yangians \cite{GRW19}, as well as to define a coproduct structure for Yangians associated with Kac–Moody algebras \cite{GNW18}.

It is natural to expect that a similar minimalistic presentation also exists for twisted Yangians. Our first main result addresses this question affirmatively. 

Let $h_{i,2r+1},b_{i,r}$ for $i\in \I$ and $r\in\bN$ be the Drinfeld generators of the twisted Yangian $\Yi$ with the defining relations \eqref{ty0}--\eqref{fSerre3}. Then $\Yi$ is generated by the elements $h_{i,1},b_{i,0}$ for $i\in \I$, see Lemma \ref{lem:generate}.

\begin{alphatheorem}[Theorem \ref{thm:min-text}]\label{thm:min}
The twisted Yangian $\Yi$ is isomorphic to the algebra generated by $h_{i,1}$, $b_{i,0}$, $b_{i,1}$ for $i\in \I$ subject to only the relations \eqref{redhh}--\eqref{redbb} together with the finite Serre type relations \eqref{fSerre0}--\eqref{fSerre3}. If $\g$ is of type $\mathsf A_1$, $\mathsf B_2\cong\mathsf C_2$, or $\mathsf G_2$, then an additional relation \eqref{eq:add-rel} should be included for any single $i\in\I$.
\end{alphatheorem}

The extra relation \eqref{eq:add-rel} for $\Yi$ is analogous to the relation $\big[\xi_{i,1},[x_{i,1}^+,x_{i,1}^-]\big]=0$ for $\Y$, which seems to be necessary when the rank is very small. More precisely, when there are at least two nodes in the Dynkin diagram, then the relation $\big[\xi_{i,1},[x_{i,1}^+,x_{i,1}^-]\big]=0$ for $\Y$ can be deduced from other relations in the minimalistic presentation. A similar situation happens when there is a subdiagram of type $\mathsf A_2$ in the Dynkin diagram for the case of twisted Yangians (cf. Lemma \ref{lem:hi4}).

Recently, a completely different Drinfeld presentation of twisted Yangian of split type $\mathsf A_{2n}$ is given in \cite{HU25}, where a corresponding minimalistic presentation is also introduced. We expect their presentation is related to the parabolic presentation in \cite{LPTTW25} with the matrix chosen and the composition $(2^n)$.

\subsection{Embedding into Yangians}
Twisted Yangian $\YiJ$ associated to arbitrary symmetric pairs was introduced in \cite{Mac02} as a coideal subalgebra of Yangian in the $J$ presentation. The same family of twisted Yangians are also obtained via a homogeneous quantization of a Lie coideal structure for twisted current algebras in \cite{BR17}, where the generators and defining relations analogous to \cite[Theorem 2]{Dri85} are explicitly described.

It is important to understand if the twisted Yangians in Drinfeld presentations are coideal subalgebras of Yangians and how twisted Yangians in Drinfeld presentations from \cite{LWZ25affine} and in $J$ presentations from \cite{Mac02,BR17} are related. Our second main result is to address this question with the help of Theorem \ref{thm:min}.

\begin{alphatheorem}[Theorem \ref{thm:embedmaintext}]\label{thm:embed}
Let $\g$ be a simple Lie algebra except for type $\mathsf G_2$. The map $\varphi$ defined by
\begin{align}
    b_{i,0}&\mapsto x_{i}^+-x_i^-,\notag\\
    h_{i,1}&\mapsto 2\xi_{i,1}-\xi_{i}^2+\sum_{\alpha\in\Phi^+}(\alpha,\alpha_i)(x_\alpha^+)^2,\label{eq:intro1}\\
    b_{i,1}&\mapsto x_{i,1}^++x_{i,1}^-+\tfrac12\sum_{\alpha\in\Phi^+}\big\{[x_{i}^+,x_{\alpha}^+],x_{\alpha}^+\big\}-\tfrac12\{x_i^+,\xi_i\}.\label{eq:intro2}
\end{align}
induces an algebra monomorphism from $\Yi$ to $\Y$, which identifies $\Yi$ as a right coideal subalgebra of $\Y$. In particular, $\Yi$ is isomorphic to the twisted Yangian $\YiJ$ in $J$ presentation. 
\end{alphatheorem}

The key part of Theorem \ref{thm:embed} is to verify that $\varphi$ induces an algebra homomorphism, which relies on Theorem \ref{thm:min}. In general the extra relation \eqref{eq:add-rel} is  complicated to verify, cf. \cite[Appendix]{GRW19} where the relation $\big[\xi_{i,1},[x_{i,1}^+,x_{i,1}^-]\big]=0$ is verified for Yangian of $\g=\fksl_2$ in terms of $J$ presentation. For split type $\mathsf A_1$, we shall use the explicit isomorphism between twisted Yangians in R-matrix  and  Drinfeld presentations from \cite{LWZ25GD}, see Appendix \ref{sec:app}. The case of type $\mathsf B_2\cong \mathsf C_2$ can be attacked similarly using the R-matrix presentation by verifying the extra relation \eqref{eq:add-rel}, see Appendix \ref{sec:app2} for more detail.

Another main challenge of Theorem \ref{thm:embed} is to find the explicit correspondence for $h_{i,1}$, $b_{i,1}$ in \eqref{eq:intro1}--\eqref{eq:intro2}. Remarkably, these formulas are surprisingly similar to the correspondence between Yangians in Drinfeld and $J$ presentations,
\begin{align}
  J(\xi_i) & \mapsto \xi_{i,1}- \tfrac12 \xi_i^2 + \tfrac14 \sum_{\alpha\in\Phi^+}
  (\alpha,\alpha_i) \{x^+_\alpha, x^-_\alpha\},\label{eq:J1}
\\
  J(x^+_i) &\mapsto x^\pm_{i, 1} \pm \tfrac14
  \sum_{\alpha\in\Phi^+} \left\{[x_i^\pm,x^\pm_\alpha],x^\mp_\alpha\right\} 
  - \tfrac14\{x^\pm_i,\xi_i\},\notag
\end{align}
see \cite[Theorem~1]{Dri87} and \cite[Theorem~2.16]{GRW19}. For instance, changing $x_\alpha^-$ to $x_\alpha^+$ in the RHS of \eqref{eq:J1} and multiplying the expression by two, one obtains the RHS of \eqref{eq:intro1}.

A different notion of ``twisted Yangians" is introduced in \cite[Definition 3.15]{LWW25} in connection with quantization of loop symmetric space via quantum duality principle \cite{Dr87b,G02}. Theorem \ref{thm:embed} gives an answer to the first two questions listed in \S\ref{sec:goal} and  proves a conjecture raised in \cite[Remark 3.23]{LWW25} that the split twisted Yangians in Drinfeld presentations are ``twisted Yangians" in the sense of \cite[Definition 3.15]{LWW25}.

\subsection{Estimates of generators}
Finally, let us approximately express the Drinfeld generators $h_{i,2r+1},b_{i,r}$ for $\Yi$ and their images under the coproduct for $\Y$ in terms of the Drinfeld generators $\xi_{i,r},x_{i,r}^\pm$ for $\Y$. Denote by $Q$ the root lattice of $\g$ and $Q_+$ the subset of $\bN$-linear combinations of simple roots (excluding zero weight). It is well known  that the Yangian $\Y$ is $Q$-graded. For a subset $\mathsf S\subset Q$, denote $\Y_{\mathsf S}$ the subspace of $\Y$ spanned by homogeneous elements of degree in $\mathsf S$.

\begin{alphatheorem}[Theorem \ref{thm:hb}]
\label{thm:hb-intro}
Let $\g$ be a simple Lie algebra except for type $\mathsf G_2$. We have
\begin{align}
&h_i(u)\equiv \xi_i(u)\xi_i(-u) &\pmod{\Y_{Q_+}[\![u^{-1}]\!]},\label{eq:h-est-intro}\\
&b_i(u)\equiv \tfrac12 \{x_i^+(u),\xi_i(-u)\}+x_i^-(-u)   &\pmod{\Y_{\alpha_i+Q_+}[\![u^{-1}]\!]},\label{eq:b-est-intro}\\
&\Delta(h_i(u))\equiv h_i(u)\otimes\xi_i(u)\xi_i(-u)  &\pmod{\Yi\otimes \Y_{Q_+}[\![u^{-1}]\!]},\label{eq:h-co-est-intro}\\
&\Delta(b_i(u))\equiv b_i(u)\otimes\xi_i(-u)+1\otimes b_i(u) &\pmod{\Yi\otimes \Y_{Q_+}[\![u^{-1}]\!]}.\notag
\end{align}
\end{alphatheorem}

Similar formulas to \eqref{eq:h-est-intro} and \eqref{eq:h-co-est-intro} for affine $\imath$quantum groups were conjectured in \cite{WZ23} and are proved in \cite[Corollary 9.16]{Prz23} for split type $\sf A$ and in \cite[Theorem 8.1]{LP25} for split type $\sf BCD$. However, the formula \eqref{eq:b-est-intro} seems to be new.

Our proof of Theorem \ref{thm:hb-intro} works uniformly for all types (assuming Theorem \ref{thm:embed} for type $\mathsf G_2$) using the explicit formulas \eqref{eq:intro1}--\eqref{eq:intro2}. We conveniently reduce the induction steps to verifying relations in generating series form, see \S\ref{sec:proof-thmC}. We deduce \eqref{eq:h-co-est-intro} from \eqref{eq:h-est-intro}, the coproduct estimates for $\xi_i(u)$, and properties of twisted Yangians, following a similar idea used in \cite{Prz23}.

An important consequence of \eqref{eq:h-est-intro} is that it allows one to calculate the joint spectrum of $h_i(u)$ acting on a $\Y$-module (whose character is known) regarded as a $\Yi$-module via restriction, see Corollary \ref{cor:cha}. 

Since the coefficients of $h_i(u)$ form a maximal commutative subalgebra of $\Yi$, one can define the ${}^\imath q$-character (or boundary $q$-character) for twisted Yangians, similarly to \cite{Kn95} for Yangians and \cite{FR99} for quantum affine algebras. The formula \eqref{eq:h-co-est-intro} is a crucial property (cf. \cite[Lemma 1]{Kn95} and \cite[Lemma 1]{FR99}) to ensure ${}^\imath q$-character map for twisted Yangians to be a $\Y$-module homomorphism, see \cite[\S9.2]{LP25} for more detail.

\subsection{Future directions}
It would be interesting to prove Theorem \ref{thm:embed} for the remaining  type $\mathsf G_2$, and then Theorem \ref{thm:hb-intro} would follow. The main difficulty is verifying the extra relation \eqref{eq:add-rel}, which already seems to be highly nontrivial for the Yangian of $\fksl_2$ (cf. \cite[Appendix]{GRW19}). We expect this case can be addressed again via the R-matrix presentation, similarly to types $\mathsf A_1$ and $\mathsf B_2\cong\mathsf C_2$.

Twisted Yangians of quasi-split type in Drinfeld presentation were introduced in \cite{LZ24} where an explicit isomorphism between twisted Yangians in Drinfeld and R-matrix presentations was established for quasi-split type $\mathsf A$ (a case of type $\mathsf{AIII}$). Therefore, an analogue of Theorem \ref{thm:embed} for quasi-split type $\mathsf A$ can be established with a similar calculation to that done in Appendix \ref{sec:app}. Furthermore, one then obtains an analogue of Theorem \ref{thm:hb-intro} for quasi-split type $\mathsf A$. Thus, it is natural to expect that the results of this article can be generalized to the quasi-split types.

One would expect our approach can be generalized to the $q$-deformed case (affine $\imath$quantum groups) so that Theorem \ref{thm:hb-intro} can be deduced uniformly for all split types and quasi-split $\mathsf{ADE}$ types. However,  the isomorphism between the Drinfeld-Jimbo \cite{Kolb14} and the Drinfeld presentations was already established in \cite{LW21,Zha22,mLWZ24} and hence this difficulty encountered in the twisted Yangian case does not appear in the affine $\imath$quantum group case. 

In \cite{FKPRW18}, a minimalistic presentation for shifted Yangians $\Y_\mu$ similar to \cite{Lev93gen} (with the extra relation included) was used to define a (shifted) coproduct homomorphism from $\Y_\mu$ to $\Y_{\mu_1}\otimes \Y_{\mu_2}$, where $\mu,\mu_1,\mu_2$ are integral coweights such that $\mu=\mu_1+\mu_2$, generalizing the baby coproduct for shifted Yangian of type $\mathsf A$ introduced in \cite{BK06}. Recently, a family of shifted twisted Yangians of quasi-split type $\Yi_\mu$ were introduced in \cite{LWW25SiY} (see also \cite{TT24,LPTTW25,SSX25}) to study the geometry of fixed point loci of affine Grassmannian slices. One expects that shifted twisted Yangians also admit a similar minimalistic presentation. In particular, such a minimalistic presentation could be used to prove that there exist various shifted coproduct homomorphisms from  $\Yi_{\mu+\nu+\tau \nu}$ to $\Yi_\mu\otimes \Y_{\nu+\tau \nu}$, where $\tau$ is an involution of $\g$. In the split type $\mathsf A$ case, such a family of shifted coproduct homomorphisms (also called baby coproduct and introduced in \cite[\S9]{LPTTW25} for dominant $\mu$ and $\nu=0$) plays a fundamental role  in connecting truncated shifted twisted Yangians with finite $\mathcal W$-algebras of classical type \cite{LPTTW25}. The classical limit of such shifted coproducts is expected to be related to the multiplication maps between affine Grassmannian slices and islices introduced in \cite{LWW25}.

\subsection{Organization} The paper is organized as follows. In the preliminary Section \ref{sec:pre}, we review Yangians and twisted Yangians in $J$ and Drinfeld presentations. Some basic properties of these algebras are recalled. In Section \ref{sec:main1}, we study the minimalistic presentation for twisted Yangians by investigating the corresponding results on the level of associated graded. In Section \ref{sec:coideal}, we apply the minimalistic presentation to embed twisted Yangians into Yangians and show that they are isomorphic to twisted Yangians in $J$ presentations. Finally, we express Drinfeld generators of twisted Yangians and their coproduct images in terms of that of Yangians in Section \ref{sec:estimate}. Appendix \ref{sec:app+} is devoted to the special cases of types $\mathsf A_1$ and $\mathsf B_2\cong \mathsf C_2$ via twisted Yangians in R-matrix presentation.

\medskip

\noindent {\bf Acknowledgement.} The author thanks J. Brundan, N. Guay, J.-R. Li, E. Mukhin, T. Przeździecki, Y. Shen, W. Wang, B. Webster and W. Zhang for helpful discussions and correspondences. 
The author was partially supported by Weiqiang Wang's NSF grant DMS--2401351 when the author was at the University of Virginia. 

% \medskip

% \noindent {\bf Data Availability.} The manuscript has no associated data.

% \medskip

% \noindent {\bf Conflict of interest.} The author declares that there are no conflicts of interest.

\section{Yangians and twisted Yangians}\label{sec:pre}
Throughout the paper, all (Lie) algebras are defined over the complex field $\mathbb C$.

\subsection{Lie algebras}
Let $\g$ be a finite-dimensional simple Lie algebra associated with the Cartan matrix $C=(c_{ij})_{i,j\in\I}$, where $\I=\{1,2,\cdots,n\}$ is the set of vertices of the Dynkin diagram corresponding to $\g$. 

Fix an invariant inner product $(\cdot,\cdot)$ on $\g$ and normalize the Chevalley generators $x_i^\pm,\xi_i$ so that $(x_i^+,x_i^-)=1$ and $\xi_i=[x_i^+,x_i^-]$. Let $\Phi$ and $\Phi^+$ be the set of roots and the set of positive roots, respectively. Denote the simple roots by $\alpha_i$ for $i\in \I$. Let $D = \mathrm{diag}(d_1,\dots,d_n)$, where
$d_i=\hf(\alpha_i,\alpha_i)$ for $i\in \I$. Then the matrix $A=DC$ is symmetric. 

Denote the root space corresponding to a root $\beta\in\Phi$ by $\g_\beta$. For each $\alpha\in \Phi^+$, let $x^\pm_\alpha\in \g_{\pm\alpha}$ be nonzero root vectors normalized so that $(x_\alpha^+,x_\alpha^-)=1$ and $x_{\alpha_i}^\pm=x_i^\pm$. We set $x_{\alpha}=0$ if $\alpha\notin \Phi$ and
\[
x_\alpha:=\begin{cases}
x_\alpha^+, & \text{ if }\alpha\in\Phi^+,\\
x_{-\alpha}^-,& \text{ if }-\alpha\in\Phi^+.
\end{cases}
\]
We also set $x_{\alpha}^\pm=0$ if $\alpha\notin\Phi^+$.

For $\alpha,\beta\in\Phi$, define the structure constant $\eta_{\alpha,\beta}$ by the rule: $\eta_{\alpha,\beta}=0$ if $\alpha+\beta\notin\Phi$ and
\[
[x_{\alpha},x_\beta]=\eta_{\alpha,\beta}x_{\alpha+\beta}
\]
if $\alpha+\beta\in \Phi$. We can further rescale the root vectors such that $\eta_{\alpha,\beta}=\eta_{-\beta,-\alpha}$ for $\alpha,\beta\in\Phi^+$, see e.g. \cite[\S25.2]{Hum72}. Therefore we have the following equalities, which we shall frequently use in \S\ref{sec:pfthmB0}:
\beq\label{eta}
\eta_{\alpha,\beta}=-\eta_{\beta,\alpha}=\eta_{-\beta,-\alpha}=\eta_{-\alpha,\alpha+\beta},
\eeq
for $\alpha,\beta\in\Phi^+$ such that $\alpha+\beta\in\Phi^+$. Here the last equality follows from the invariant property of the bilinear form.

Let $\omega$ be the Chevalley involution of $\g$ defined by
\beq\label{omega}
\omega:\g\to\g,\qquad \xi_i\to -\xi_i,\qquad x_i^\pm \to -x_i^\mp,\qquad (i\in \I).
\eeq
Denote by $\mathfrak k:=\g^\omega$ the $\omega$-fixed point Lie subalgebra of $\g$ and by $\mathfrak p$ the eigenspace of $\omega$ corresponding to the eigenvalue $-1$. Then $(\g,\mathfrak k)$ forms a symmetric pair of \textit{split type}. 

Note that 
\beq\label{eq:kprel}
[\mathfrak k,\mathfrak k]\subset \mathfrak k,\qquad [\mathfrak k,\mathfrak p]\subset \mathfrak p,\qquad [\mathfrak p,\mathfrak p]\subset \mathfrak k,
\eeq
and $\omega(x_\alpha^\pm)=-x_\alpha^\mp$ for $\alpha\in\Phi^+$. A basis for $\mathfrak k$ is given by $x_{\alpha}^+-x_\alpha^-$ for $\alpha\in\Phi^+$ while a basis for $\mathfrak p$ is given by $x_{\alpha}^++x_\alpha^-$ and $\xi_i$ for $\alpha\in\Phi^+$ and $i\in \I$. Set
\beq\label{by-def}
b_{\alpha}:=x_{\alpha}^+-x_\alpha^-,\qquad  y_{\alpha}:=x_{\alpha}^++x_\alpha^-,\qquad (\alpha\in\Phi^+).
\eeq

Denote by $C_{\mathfrak k}$, $C_{\mathfrak p}$, and $C_{\g}$ the Casimir elements of $\mathfrak k$, $\mathfrak p$, and $\g$, respectively. We have
\beq\label{Ck}
C_{\mathfrak k}=-\tfrac12\sum_{\alpha\in\Phi^+}(x_{\alpha}^+-x_{\alpha}^-)^2.
\eeq
Denote by $\Omega_{\mathfrak k}$, $\Omega_{\mathfrak p}$, and $\Omega_{\g}$ the 2-tensor Casimir elements of $\mathfrak k$, $\mathfrak p$, and $\g$, respectively. We have
\beq\label{Omega}
\Omega_{\g}= \Omega_{\mathfrak k}+\Omega_{\mathfrak p}, \qquad [x\otimes 1+1\otimes x,\Omega_\g]=0,
\eeq
for $x\in\g$.
\begin{lem}\label{lem:omega}
For $x\in\mathfrak p$, we have
\[
[x\otimes 1,\Omega_{\mathfrak p}]+[1\otimes x,\Omega_{\mathfrak k}]=[x\otimes 1,\Omega_{\mathfrak k}]+[1\otimes x,\Omega_{\mathfrak p}]=0.
\]
\end{lem}
\begin{proof}
It follows from \eqref{Omega} that
\[
[x\otimes 1,\Omega_{\mathfrak p}]+[1\otimes x,\Omega_{\mathfrak k}]=-[x\otimes 1,\Omega_{\mathfrak k}]-[1\otimes x,\Omega_{\mathfrak p}].
\]
Since $x\in\mathfrak p$, it follows from \eqref{eq:kprel} that the LHS belongs to $\mathfrak k\otimes \mathfrak p$ while the RHS belongs to $\mathfrak p\otimes \mathfrak k$. Therefore, the statement follows.
\end{proof}
\begin{rem}\label{rem:inv}
Clearly, Lemma \ref{lem:omega} holds true for arbitrary involution $\omega$ as it relies on \eqref{eq:kprel} which is clearly true for any involution $\omega$ of $\g$.
\end{rem}

\subsection{Yangians}\label{sec:Yang}
We recall definitions of Yangians, following \cite{GNW18}. The Yangian $\Y$ associated to an arbitrary finite-dimensional simple Lie algebra $\g$ was first introduced by Drinfeld in the $J$ presentation.
\begin{dfn}[\cite{Dri85}]
\label{def:Y-J}
The Yangian $\Y:=\Y(\g)$ is the unital associative algebra generated by elements $x$ and $J(x)$ for $x\in\g$ with the defining relations:
\begin{equation}\label{eq:another}
\begin{split}
&xy-yx = [x,y] \text{ for all } x,y\in\g,  \text{ $J$ is linear in }  x\in\g, \quad
   J([x,y]) = [x, J(y)],
\\
  &\big[J(x), J([y,z])\big] + \big[J(z), J([x,y])\big] + \big[J(y), J([z,x])\big]\\
  &\hskip3.5cm = \sum_{a,b,c\in \Lambda} \big([x,\xi_a],\big[[y,\xi_b],[z,\xi_c]\big]\big) \big\{ \xi_a,\xi_b,\xi_c \big\},
\\
  &\big[[J(x), J(y)], [z, J(w)]\big] + \big[[J(z), J(w)], [x, J(y)]\big]\\
 &=  \sum_{a,b,c\in \Lambda} \Big(\big([x,\xi_a],\big[[y,\xi_b],[[z,w],\xi_c]\big]\big) \\&
 \hskip3.5cm + \big([z,\xi_a],\big[[w,\xi_b],[[x,y],\xi_c]\big]\big) \Big)\big\{ \xi_a,\xi_b,J(\xi_c) \big\},
\end{split}
\end{equation}
where $\{ \xi_a \}_{a \in \Lambda}$ is an orthonormal basis of $\g$, $\Lambda$ being a fixed indexing set of size $\dim \g$, and $\{ \xi_a,\xi_b,\xi_c \} = \frac{1}{24} \sum_{\pi\in \fkS_3} \xi_{\pi(a)} \xi_{\pi(b)} \xi_{\pi(c)}$, $\fkS_3$ being the group of permutations of $\{ a,b,c\}$.
\end{dfn}

The Yangian $\Y$ is a Hopf algebra with the coproduct determined by
\beq\label{copro-J}
\Delta(x)=x\otimes 1+1\otimes x,\quad \Delta(J(x))=J(x)\otimes 1+1\otimes J(x)+\tfrac12[x\otimes 1,\Omega_\g].
\eeq

Later, Drinfeld introduced a new presentation of $\Y$, now known as the Drinfeld (new/current) presentation, which is described below.

Given two elements $x,y$ in an associative algebra $\mathcal A$, set $\{x,y\}:=xy+yx$.

\begin{dfn}[\cite{Dri87}]
\label{def:Yangian}The Yangian $\Y$ is the unital associative algebra with generators $x_{i
  ,r}^\pm,\xi_{i,r}$ for $i\in \I, r\in\bN$ subject to the following
defining relations: 
\begin{align}
  [\xi_{i,r}, \xi_{j,s}]& = 0, \label{eq:relHH}\\
  [\xi_{i,0}, x^\pm_{j,s} ]& = \pm (\alpha_i,\alpha_j) x^\pm_{j,s},
  \label{eq:relHX}
\\
  [x^+_{i,r}, x_{j,s}^-]& = \delta_{ij} \xi_{i, r+s}, \label{eq:relXX}
\\
  [\xi_{i, r+1}, x^\pm_{j,s}] - [\xi_{i,r}, x^\pm_{j, s+1}]& =
  \pm\tfrac12(\alpha_i,\alpha_j) \big\{
    \xi_{i,r},x^\pm_{j,s}
  \big\}, \label{eq:relexHX}
\\
  [x^\pm_{i, r+1}, x^\pm_{j, s}] - [x^\pm_{i, r}, x^\pm_{j, s+1}] &= 
  \pm\tfrac12(\alpha_i,\alpha_j) \big\{
    x^\pm_{i,r},x^\pm_{j,s}\big\},
   \label{eq:relexXX}
\\
  \sum_{\sigma\in \fkS_m}
   \Big[x^\pm_{i, r_{\sigma(1)}}, \big[ x^\pm_{i, r_{\sigma(2)}}, \cdots,\,
       &[x^\pm_{i, r_{\sigma(m)}}, x^\pm_{j, s}] \cdots \big]\Big]
   = 0
   \quad \text{if $i\neq j$,}
 \label{eq:relDS}
\end{align}
where $m = 1 - c_{ij}$.
\end{dfn}

Set
\beq\label{tlxidef}
\tl \xi_{i,1}:=\xi_{i,1}-\tfrac12\xi_{i,0}^2,
\eeq 
then by \eqref{eq:relHX} and \eqref{eq:relexHX} we have
\beq\label{tlxicom}
\big[\tl \xi_{i,1},x_{j,r}^\pm\big]=\pm(\alpha_i,\alpha_j)x_{j,r+1}^\pm.
\eeq 

The universal enveloping algebra $\mathrm{U}(\g)$ is identified as a subalgebra of $\Y$ via the map $\xi_i\mapsto \xi_{i,0}$, $x_i^{\pm}\mapsto x_{i,0}^\pm$ for $i\in\I$. 

The isomorphism between this presentation and the one provided in Definition \ref{def:Yangian} is given by (see \cite[Theorem~1]{Dri87} and \cite[Theorem~2.16]{GRW19})
\begin{equation}\label{eq:J}
  \begin{split}
  x^\pm_i & \mapsto x^\pm_{i,0}, \hskip1.7cm \xi_i \mapsto \xi_{i,0},
\\
  J(\xi_i) & \mapsto \xi_{i,1} + v_i, \hskip0.9cm
  v_i := \tfrac14 \sum_{\alpha\in\Phi^+}
  (\alpha,\alpha_i) \{x^+_\alpha, x^-_\alpha\}
  - \tfrac12 \xi_i^2,
\\
  J(x^\pm_i) &\mapsto x^\pm_{i, 1} + w^\pm_i, \hskip0.6cm
  w^\pm_i := \pm \tfrac14
  \sum_{\alpha\in\Phi^+} \left\{[x_i^\pm,x^\pm_\alpha],x^\mp_\alpha\right\} 
  - \tfrac14\{x^\pm_i,\xi_i\}.
  \end{split}
\end{equation}
In terms of the Drinfeld presentation, the coproduct defined in \eqref{copro-J} satisfies
\begin{align}
&\Delta(\xi_i)=\xi_i\otimes 1+1\otimes \xi_i,\qquad ~~~~~\Delta(x_i^\pm)=x_i^\pm\otimes 1+1\otimes x_i^\pm,\label{copro-efh}\\
&\Delta(\xi_{i,1})=\xi_{i,1}\otimes 1+1\otimes \xi_{i,1}+\xi_{i}\otimes \xi_{i}-\sum_{\alpha\in\Phi^+}(\alpha,\alpha_i)x_{\alpha}^-\otimes x_\alpha^+,\label{coproxii1}\\
&\Delta(x_{i,1}^+)=x_{i,1}^+\otimes 1+1\otimes x_{i,1}^++\xi_{i}\otimes x_{i}^+-\sum_{\alpha\in\Phi^+}x_{\alpha}^-\otimes [x_{i}^+,x_\alpha^+],\\
&\Delta(x_{i,1}^-)=x_{i,1}^-\otimes 1+1\otimes x_{i,1}^-+x_{i}^-\otimes \xi_{i}+\sum_{\alpha\in\Phi^+}[x_{i}^-,x_\alpha^-]\otimes x_\alpha^+.\label{copro-xi-}
\end{align}

The coproduct for Drinfeld generators of higher degrees has the following estimates. 

Let $\Y^{\gge 0}$ (resp. $\Y^{\lle 0}$) be the Borel subalgebra of $\Y$ generated by the elements $\xi_{i,r}$, $x_{i,r}^+$ (resp $x_{i,r}^-$) for $i\in\I$ and $r\in\bN$. Let $Q$ be the root lattice. Set
\[
Q_+:=\Big\{\sum_{i\in\I}k_i\alpha_i~\Big|~k_i\in\bN,\sum_{i\in \I}k_i>0\Big\},\qquad Q_-:=-Q_+.
\]
It is well known that $\Y$ is $Q$-graded by setting
\[
\mathrm{deg}\,\xi_{i,r}=0,\qquad \mathrm{deg}\,x_{i,r}^\pm=\pm\alpha_i. 
\]
For $\alpha\in Q$ and $\mathsf S\subset Q$, denote
\[
\Y_{\alpha} :=\{w\in \Y\mid \deg w=\alpha\},\qquad \Y_{\mathsf S} :=\mathrm{span}\{w\in \Y\mid \deg w\in \mathsf S\}.
\]
In particular, $\Y_{Q_+}$ is the subalgebra of $\Y$ spanned by homogeneous elements of degree in $Q_+$. Similarly, one can define these notations for the subalgebras $\Y^{\gge 0}$ and $\Y^{\lle 0}$.

Set
\beq\label{xix_iu}
\xi_i(u)=1+\sum_{r\gge 0}\xi_{i,r}u^{-r-1},\qquad  x_i^\pm(u)=\sum_{r\gge 0}x_{i,r}^\pm u^{-r-1}.
\eeq

The following well-known statement can be found in \cite[Proposition 2.9]{GW23} and {\cite[Lemma 2.5]{HZ24}}. This is an improvement of \cite[Lemma 1]{Kn95} and \cite[Proposition 2.8]{CP91}.
\begin{lem}\label{lem:copro}
We have
\begin{align*}
&\Delta(\xi_i(u))\equiv \xi_i(u)\otimes \xi_i(u) &\pmod{\Y^{\lle 0}_{Q_-}\otimes \Y^{\gge0}_{Q_+}[\![u^{-1}]\!]},\\
&\Delta(x_i^+(u)) \equiv x_i^+(u)\otimes 1+\xi_i(u)\otimes x_i^+(u) &\pmod{\Y^{\lle 0}_{Q_-}\otimes \Y^{\gge0}_{\alpha_i+Q_+}[\![u^{-1}]\!]},\\
&\Delta(x_i^-(u)) \equiv x_i^-(u)\otimes \xi_i(u)+1\otimes x_i^-(u) &\pmod{\Y^{\lle 0}_{-\alpha_i+Q_-}\otimes \Y^{\gge0}_{Q_+}[\![u^{-1}]\!]}.
\end{align*}
\end{lem}

Define a filtration on $\Y$ by setting $\mathrm{deg}\,\xi_{i,r}=\mathrm{deg}\,x_{i,r}^\pm=r$ and denote by $\gr\,\Y$ the associated graded algebra of $\Y$. Let $\bar\xi_{i,r},\bar x_{i,r}^\pm$ be the images for of $ \xi_{i,r}, x_{i,r}^\pm$ in the $r$-th component of $\mathrm{gr}\,\Y$, respectively. Then it is well known that there exists a graded Hopf algebra isomorphism
\beq\label{ass}
\begin{split}
\rho:~&\mathrm{U}(\g[z])\stackrel{\cong\,}{\longrightarrow} \mathrm{gr}\,\Y, \\
&\xi_iz^r\mapsto \bar\xi_{i,r},\qquad x_i^{\pm }z^r \mapsto \bar x_{i,r}^\pm,\qquad  (i\in \I,r\in\bN).
\end{split}
\eeq

The Yangian $\Y$ has an anti-involution $\mathfrak T$ given by
\beq\label{eq:tau}
\mathfrak T:\Y\to\Y,\quad \xi_{i,r}\mapsto \xi_{i,r},\quad x_{i,r}^\pm \mapsto x_{i,r}^\mp.
\eeq
Clearly, we have 
\beq\label{eq:tauJ}
\mathfrak T(J(\xi_i))=J(\xi_i),\qquad \mathfrak T (J(x_i^\pm))=J(x_i^\mp).
\eeq
Moreover,
\beq\label{eq:tau-copro}
\Delta \circ \mathfrak T =(\mathfrak T\otimes \mathfrak T)\circ \Delta^{\mathrm{op}},
\eeq
where $ \Delta^{\mathrm{op}}$ is the opposite coproduct.

\subsection{Twisted Yangians}\label{sec:ty}
Now let us recall the twisted Yangians in $J$ presentation from \cite{Mac02,BR17} and twisted Yangians of split types in Drinfeld presentation from \cite{LWZ25affine}. 

Twisted Yangians for arbitrary symmetric pairs were defined as coideal subalgebras of $\Y$ via $J$ presentation in \cite{Mac02}. It has been further studied in \cite{BR17} via homogeneous quantization, where an explicit set of generators with defining relations similar to Definition \ref{def:Y-J} was obtained.
\begin{dfn}[\cite{Mac02,BR17}]\label{def:tY-J}
The twisted Yangian $\YiJ$ in the $J$ presentation is the subalgebra of $\Y$ generated by all the elements
\[
x,\quad B(y):=J(y)-\tfrac14[y,C_{\mathfrak k}],
\]
where $x\in \mathfrak k$ and $y\in\mathfrak p$. 
\end{dfn}
Note that $\YiJ$ is a \textit{right} coideal subalgebra of $\Y$, i.e. $\Delta(\YiJ)\subset \YiJ\otimes \Y$. Indeed, it follows from Lemma  \ref{lem:omega} and \eqref{copro-J} that
\begin{align*}
&\Delta(x)=x\otimes 1+1\otimes x,\\
&\Delta(B(y))=B(y)\otimes 1+1\otimes B(y)-[1\otimes y,\Omega_{\mathfrak k}].
\end{align*}
Moreover, this definition works for an arbitrary involution $\omega$ (cf. Remark \ref{rem:inv}) and hence can be used to define twisted Yangians associated to arbitrary symmetric pairs.

In \cite[Theorem~5.5]{BR17}, the definition of the twisted Yangian in its $J$ presentation is slightly different. Explicitly, the twisted Yangian $\Yi_{\mathscr J}$ is the subalgebra of $\Y$ generated by all the elements 
\[
x,\quad \mathscr B(y):=J(y)+\tfrac14[y,C_{\mathfrak k}],
\]
where $x\in \mathfrak k$ and $y\in\mathfrak p$. Namely, $B(y)$ is replaced by $\mathscr B(y)$ (the plus is changed to minus). Consequently, the twisted Yangian $\Yi_{\mathscr J}$ is a \textit{left} coideal subalgebra of $\Y$ as
\[
\Delta(\mathscr B(y))=\mathscr B(y)\otimes 1+1\otimes \mathscr B(y)+[y\otimes 1,\Omega_{\mathfrak k}]\in \Y\otimes \Yi_{\mathscr J}.
\]
The two different subalgebras $\YiJ$ and $\Yi_{\mathscr J}$ of the Yangian $\Y$ are related by the anti-involution $\mathfrak T$ defined in \eqref{eq:tau}: $\mathfrak T(\YiJ)=\Yi_{\mathscr J}$. This can be easily seen from \eqref{eq:tauJ} and \eqref{eq:tau-copro}.

Recently, a Drinfeld presentation for twisted Yangians of split types was proposed in \cite{LWZ25affine}. Recall that $d_i=\hf (\alpha_i,\alpha_i)$ for $i\in \I$.

\begin{dfn}[{\cite[Definition~4.1]{LWZ25affine}}]\label{def:YN}
The twisted Yangian $\Yi$ in Drinfeld presentation is the unital associative algebra generated by $h_{i,2r+1},b_{i,r}$, for $i\in \I$ and $r\in\bN$, subject to the following relations \eqref{ty0}--\eqref{ty2}:
\begin{align}
\label{ty0}
[h_{i,r},h_{j,s}] &=0,\\
\label{ty1}
[h_{i,r+1},b_{j,s}]-[h_{i,r-1},b_{j,s+2}]
&=(\alpha_i,\alpha_j) \{ h_{i,r-1},b_{j,s+1}\}+ \tfrac14(\alpha_i,\alpha_j)^2[h_{i,r-1},b_{j,s}],
\\
\label{ty2}
[b_{i,r+1},b_{j,s}]-[b_{i,r},b_{j,s+1}]&=\tfrac12(\alpha_i,\alpha_j) \{ b_{i,r},b_{j,s}\}-2\delta_{ij} (-1)^s h_{i,r+s+1},
\end{align}
and the finite Serre type relations \eqref{fSerre0}--\eqref{fSerre3},
\begin{align}
\label{fSerre0}
 [b_{i,0},b_{j,0}] 
&=0 &(c_{ij}=0),
\\
\label{fSerre1}
\mathrm{ad}(b_{i,0})^{2}(b_{j,0})
&=- d_ib_{j,0} &(c_{ij}=-1),
\\
\label{fSerre2}
\mathrm{ad}(b_{i,0})^{3}(b_{j,0})
 &= - 4d_i [b_{i,0},b_{j,0}] & (c_{ij}=-2),\\
\mathrm{ad}(b_{i,0})^{4}(b_{j,0})
&= -10 d_i\big[b_{i,0},[b_{i, 0},b_{j,0}]\big] -9d_i^2 b_{j,0} & (c_{ij}=-3).\label{fSerre3}
\end{align}
where $h_{i,-1}=1$, $h_{i,r}=0$ if $r$ is even or $r<-1$. 
\end{dfn}

Note that we normalized the generators from \cite{LWZ25affine} so that the presentation is parallel to the Drinfeld presentation of $\Y$ in Definition \ref{def:Yangian}.

One of our main results (Theorem \ref{thm:embedmaintext}) will show that $\Yi$ is isomorphic to $\YiJ$ and hence a right coideal subalgebra of $\Y$, justifying $\Yi$ as a ``twisted Yangian".

\medskip

We recall some properties of the twisted Yangians that will be used later.

\begin{lem}\label{lem:generate}
The twisted Yangian $\Yi$ is generated by $h_{i,1},b_{i,0}$ for $i\in \I$.
\end{lem}
\begin{proof}
Indeed, from \eqref{ty1} by setting $r=0$, we have
\beq\label{eq:hi1bjr-new}
[h_{i,1},b_{j,r}]=2(\alpha_i,\alpha_j)b_{j,r+1},
\eeq
which implies all $b_{j,r}$ for $j\in \I$ and $r\in\bN$ can be generated by $h_{i,1},b_{i,0}$ for $i\in \I$. Then it follows from \eqref{ty2} that all $h_{i,2r+1}$ can also be generated by $h_{i,1},b_{i,0}$ for $i\in \I$.  
\end{proof}

\begin{lem}\label{lem:hbhb}
The following relation holds in $\Yi$,
\[
\Big[h_{i,1},\big[b_{i,1},[h_{i,1},b_{i,1}]\big]\Big]=(\alpha_i,\alpha_i)^2\big[b_{i,1}^2,h_{i,1}\big].
\]
\end{lem}
\begin{proof}
By \eqref{ty2}, we have
\[
[b_{i,2},b_{i,1}]-[b_{i,1},b_{i,2}]=(\alpha_i,\alpha_i)b_{i,1}^2 +2h_{i,3}.
\]
Therefore,
\beq\label{bi1bi2}
[b_{i,1},b_{i,2}]=-\hf (\alpha_i,\alpha_i)b_{i,1}^2-h_{i,3}.
\eeq
Now, we have 
\begin{align*}
\Big[h_{i,1},\big[b_{i,1},[h_{i,1},b_{i,1}]\big]\Big]&\stackrel{\eqref{eq:hi1bjr-new}}{=}2(\alpha_i,\alpha_i)\big[h_{i,1},[b_{i,1},b_{i,2}]\big]\\
&\stackrel{\eqref{bi1bi2}}{=}2(\alpha_i,\alpha_i)\big[h_{i,1},-\hf (\alpha_i,\alpha_i)b_{i,1}^2-h_{i,3}\big]\\
&\stackrel{\eqref{ty0}}{=}(\alpha_i,\alpha_i)^2\big[b_{i,1}^2,h_{i,1}\big].\qedhere
\end{align*}
\end{proof}

Extend the involution $\omega$ on $\g$ defined in \eqref{omega} to an involution $\check\omega $ on the current algebra $\g[z]$ by
\[
\check\omega: \g[z]\to \g[z],\qquad xz^r\mapsto \omega(x)(-z)^r,
\]
for $x\in \g$ and $r\in\bN$. We call $\g[z]^{\check\omega}$ a \textit{twisted current algebra}. The twisted current algebra $\g[z]^{\check\omega}$ has a basis given by
\[
\xi_iz^{2r+1},\qquad \big(x_\alpha^+-(-1)^rx_\alpha^-\big)z^r,\qquad (\alpha\in\Phi^+,r\in\bN).
\]

Define a filtration on $\Yi$ by setting $\mathrm{deg}\,h_{i,2r+1}=2r+1$ and $\mathrm{deg}\,x_{i,r}^\pm=r$. Denote by $\gr\,\Yi$ the associated graded algebra of $\Yi$. Let $\bar h_{i,2r+1},\bar b_{i,r}$ be the images of $ h_{i,2r+1}, b_{i,r}$ in the $(2r+1)$-st and $r$-th components of $\mathrm{gr}\,\Yi$, respectively. Then it is  known \cite[Corollary~4.13]{LWZ25affine} that there exists an algebra isomorphism
\beq\label{assi}
\begin{split}
\varrho:~&\mathrm{U}(\g[z]^{\check\omega})\stackrel{\cong\,}{\longrightarrow} \mathrm{gr}\,\Yi,\\
&2\xi_iz^{2r+1}\mapsto \bar h_{i,2r+1},\quad (x_i^+-(-1)^rx_i^-)z^r \mapsto \bar b_{i,r},\quad (i\in \I,r\in\bN).
\end{split}
\eeq
One consequence of the results proven later (e.g. Theorem \ref{thm:embedmaintext}) is that $\varrho$ can be upgraded to an algebra isomorphism of Hopf algebras.

\section{A Minimalistic presentation}\label{sec:main1}
\subsection{A Minimalistic presentation for $\Yi$}
Our first main result is the minimalistic presentation of twisted Yangian $\Yi$ in the spirit of \cite{Lev93gen,GNW18}.
\begin{thm}\label{thm:min-text}
If $\g$ is not of type $\mathsf A_1$, $\mathsf B_2\cong\mathsf C_2$, or $\mathsf G_2$, then the twisted Yangian $\Yi$ is isomorphic to the unital associative algebra generated by $\{h_{i,1},b_{i,0},b_{i,1}\}_{i\in \I}$, subject only to the relations:
\begin{align}
[h_{i,1},h_{j,1}]&=0,\label{redhh}\\
[h_{i,1},b_{j,0}]&=2(\alpha_i,\alpha_j)b_{j,1},\label{redhb}\\
[b_{i,1},b_{j,0}]-[b_{i,0},b_{j,1}]&=\tfrac12(\alpha_i,\alpha_j) \{b_{i,0},b_{j,0}\}-2\delta_{ij}h_{i,1},\label{redbb}
\end{align}
together with the finite Serre type relations \eqref{fSerre0}--\eqref{fSerre3}. If $\g$ is of type $\mathsf A_1$, $\mathsf B_2\cong\mathsf C_2$, or $\mathsf G_2$, then an additional relation
\beq\label{eq:add-rel}
\Big[h_{i,1},\big[b_{i,1},[h_{i,1},b_{i,1}]\big]\Big]=(\alpha_i,\alpha_i)^2\big[b_{i,1}^2,h_{i,1}\big]
\eeq
should be included for any single $i\in\I$.
\end{thm}

The theorem will be proved in Section \ref{sec:pfthmA}. 

We remark that the additional relation \eqref{eq:add-rel} is analogous to \cite[(1.6)]{Lev93gen} for Yangians. Indeed, the relation \cite[(1.6)]{Lev93gen} essentially corresponds to the relation $[\xi_{i,1},\xi_{i,2}]=0$ in the Yangian $\Y$, which turns out to be redundant except for the case $\g=\mathfrak{sl}_2$, see \cite[\S3]{GNW18}. Similarly, from the proof of Lemma \ref{lem:hbhb}, the additional relation \eqref{eq:add-rel} is essentially the relation $[h_{i,1},h_{i,3}]=0$ (recall that $h_{i,2}=0$) in the twisted Yangian $\Yi$. We shall prove that the  relation \eqref{eq:add-rel} can be dropped when there is a subdiagram of type $\mathsf A_2$, see Lemma \ref{lem:hi4} below.

\subsection{Classical picture}
Instead of working directly on the Yangian as in \cite{Lev93gen,GNW18}, we exploit the idea used to prove  \cite[Theorem~4.14]{LWZ25affine} which greatly simplifies the problem by reducing it to the level of associated graded.

\begin{prop}\label{prop:red}
If $\g$ is not of type $\mathsf A_1$, $\mathsf B_2\cong\mathsf C_2$, or $\mathsf G_2$, then the algebra $\mathrm{U}(\g[z]^{\check\omega})$ is isomorphic to the unital associative algebra generated by the elements $\bh_{i,1}$, $\bb_{i,0}$,  $\bb_{i,1}$ for $i\in\I$ subject to only the relations:
\begin{align}
[\bh_{i,1},\bh_{j,1}]&=0,\label{tchh}\\
[\bh_{i,1},\bb_{j,0}]&=2(\alpha_i,\alpha_j)\bb_{j,1},\label{tchb}\\
[\bb_{i,1},\bb_{j,0}]&=[\bb_{i,0},\bb_{j,1}]-2\delta_{ij}\bh_{i,1},\label{tcbb}\\
\label{tcfSerre0}
 [\bb_{i,0},\bb_{j,0}] 
&=0 & (c_{ij}=0),
\\
\label{tcfSerre1}
\mathrm{ad}(\bb_{i,0})^{2}(\bb_{j,0})
&=- d_i\bb_{j,0} & (c_{ij}=-1),
\\
 \label{tcfSerre2}
\mathrm{ad}(\bb_{i,0})^{3}(\bb_{j,0})
 &= - 4d_i [\bb_{i,0},\bb_{j,0}] & (c_{ij}=-2),\\
\mathrm{ad}(\bb_{i,0})^{4}(\bb_{j,0})
&= -10 d_i\big[\bb_{i,0},[\bb_{i, 0},\bb_{j,0}]\big] -9d_i^2 \bb_{j,0}& (c_{ij}=-3).\label{tcfSerre3}
\end{align}
If $\g$ is of type $\mathsf A_1$, $\mathsf B_2\cong\mathsf C_2$, or $\mathsf G_2$, then an additional relation 
\beq\label{eq:add-rel-gr}
\Big[\bh_{i,1},\big[\bb_{i,1},[\bh_{i,1},\bb_{i,1}]\big]\Big]=0
\eeq
should be included for any single $i\in\I$.
\end{prop}

We prove Proposition \ref{prop:red} in Section \ref{sec:pf-prop}. 

\subsection{Proof of Proposition \ref{prop:red}}\label{sec:pf-prop}
We first recall the following lemma from \cite[Proposition~2.9]{LWZ25affine}, which is closely related to Definition \ref{def:YN} and \eqref{assi}.
\begin{lem}\label{lem:full-gr}
The algebra $\mathrm{U}(\g[z]^{\check\omega})$ is isomorphic to the unital associative algebra generated by the elements $\sfb_{i,r}$ and $\sfh_{i,2r+1}$ for $i\in\I$ and $r\in\bN$ subject to the relations:
\begin{align}
[\sfh_{i,2r+1},\sfh_{j,2s+1}]&=0,\label{tchhf}\\
[\sfh_{i,2r+1},\sfb_{j,s}]&=2(\alpha_i,\alpha_j)\sfb_{j,2r+s+1},\label{tchbf}\\
[\sfb_{i,r+1},\sfb_{j,s}]-[\sfb_{i,r},\sfb_{j,s+1}]&=-\delta_{ij}\big((-1)^r+(-1)^s\big)\sfh_{i,r+s+1},\label{tcbbf}\\
\label{tcfSerre0f}
 [\sfb_{i,0},\sfb_{j,0}] 
&=0 & (c_{ij}=0),
\\
\label{tcfSerre1f}
\mathrm{ad}(\sfb_{i,0})^{2}(\sfb_{j,0})
&=- d_i\sfb_{j,0} & (c_{ij}=-1),
\\
\label{tcfSerre2f}
\mathrm{ad}(\sfb_{i,0})^{3}(\sfb_{j,0})
 &= - 4d_i [\sfb_{i,0},\sfb_{j,0}] & (c_{ij}=-2),\\
\mathrm{ad}(\sfb_{i,0})^{4}(\sfb_{j,0})
&= -10 d_i\big[\sfb_{i,0},[\sfb_{i, 0},\sfb_{j,0}]\big] -9d_i^2 \sfb_{j,0}& (c_{ij}=-3),\label{tcfSerre3f}
\end{align}
where $\sfh_{i,2r}=0$ for $r\in\bN$. Here the elements $\sfb_{i,r}$ and $\sfh_{i,2r+1}$ are identified as elements of $\g[z]^{\check\omega}$ by
\beq\label{eq:hb-explicit}
\sfh_{i,2r+1} \mapsto 2\xi_iz^{2r+1},\quad \sfb_{i,r}\mapsto (x_i^+-(-1)^rx_i^-)z^r.
\eeq
\end{lem}

We have the following simple observation.
\begin{lem}\label{lem:gener}
We have $\sfb_{i,r+1}=\hf(\alpha_i,\alpha_i)^{-1}[\sfh_{i,1},\sfb_{i,r}]$ and $\sfh_{i,2r+1}=[\sfb_{i,0},\sfb_{i,2r+1}]$. In particular, the algebra $\mathrm{U}(\g[z]^{\check\omega})$ is generated by $\sfb_{i,0}$, $\sfb_{i,1}$, and $\sfh_{i,1}$.
\end{lem}
\begin{proof}
The first equality follows from \eqref{tchbf}. Let us consider the second one. Observe from \eqref{tcbbf} that
\[
[\sfb_{i,r+1},\sfb_{i,r}]-[\sfb_{i,r},\sfb_{i,r+1}]=-2(-1)^r\sfh_{i,2r+1}
\]
which implies further that 
\[
[\sfb_{i,r+1},\sfb_{i,r}]=(-1)^{r+1}\sfh_{i,2r+1}.
\]
Then again by \eqref{tcbbf} it is easy to see that
\beq\label{bbhelper}
[\sfb_{i,r+s+1},\sfb_{i,r-s}]=(-1)^{r+s+1}\sfh_{i,2r+1}
\eeq
for $0\lle s\lle r$. Taking $s=r$ in this identity gives the desired equation.
\end{proof}

By \eqref{tchhf}--\eqref{tchbf} and \eqref{bbhelper}, one easily finds that
\beq\label{extra-verified}
\Big[\sfh_{i,1},\big[\sfb_{i,1},[\sfh_{i,1},\sfb_{i,1}]\big]\Big]=0.
\eeq

Let $\mathcal U$ be the algebra generated by $\bb_{i,0}$, $\bb_{i,1}$, and $\bh_{i,1}$ subject to the relations \eqref{tchh}--\eqref{tcfSerre3} and \eqref{eq:add-rel-gr} when $\g$ is of type $\mathsf A_1$, $\mathsf B_2\cong \mathsf C_2$, or $\mathsf G_2$. Then it follows from Lemma \ref{lem:full-gr} and \eqref{extra-verified} that there exists an algebra homomorphism
\beq\label{eq:psi}
\begin{split}
\psi:~&\mathcal U\to \mathrm{U}(\g[z]^{\check\omega})\\
& \bh_{i,1}\mapsto \sfh_{i,1},\qquad \bb_{i,r}\mapsto \sfb_{i,r} \qquad (r=0,1),
\end{split}
\eeq
which is surjective by Lemma \ref{lem:gener}. 

Define
\beq\label{bhdef}
\bb_{i,r+1}=\hf(\alpha_i,\alpha_i)^{-1}[\bh_{i,1},\bb_{i,r}],\quad \bh_{i,r+1}=[\bb_{i,0},\bb_{i,r+1}]\qquad (r>0).
\eeq
The relation \eqref{tcbb} implies that 
\beq\label{bi0bi1=hi1}
[\bb_{i,0},\bb_{i,1}]=\bh_{i,1}.
\eeq
Hence \eqref{bhdef} also holds true if $r=0$.
To prove $\psi$ is an isomorphism, it suffices to prove that $\bh_{i,2r}=0$ and \eqref{tchhf}--\eqref{tcbbf} can be deduced from \eqref{tchh}--\eqref{tcfSerre3} and \eqref{eq:add-rel-gr} when $\g$ is of type $\mathsf A_1$, $\mathsf B_2\cong\mathsf C_2$, or $\mathsf G_2$. This should justify the existence of the inverse homomorphism of $\psi$. We establish this result by adapting the methods of \cite{Lev93gen,GNW18}, partitioning the proof into three main steps for clarity.

\subsubsection*{\bf Step 1} We first derive some easy relations that will be used later.

Using an easy induction on $r$ by applying $[\bh_{j,1},\cdot\,]$ to \eqref{tchb} recursively together with \eqref{tchh} and \eqref{bhdef}, we obtain
\beq\label{hi1bjr}[\bh_{i,1},\bb_{j,r}]=2(\alpha_i,\alpha_j)\bb_{j,r+1} 
\eeq
for $i,j\in \I$ and $r\in\bN$.

Applying $[\bh_{i,1},\cdot\,]$ to \eqref{bi0bi1=hi1}, we obtain from \eqref{bhdef} that 
\beq\label{bi0bi2=0}
\bh_{i,2}=[\bb_{i,0},\bb_{i,2}]=0.
\eeq
Similarly, applying $[\bh_{i,1},\cdot\,]$ to $[\bb_{i,0},\bb_{i,2}]=0$, we find that
\beq\label{bi1bi2=hi3}
[\bb_{i,0},\bb_{i,3}]+[\bb_{i,1},\bb_{i,2}]=0\stackrel{\eqref{bhdef}}{\Longrightarrow } [\bb_{i,1},\bb_{i,2}]=-\bh_{i,3}.
\eeq

\begin{lem}\label{lem:ext}
Assume the relations \eqref{tchh}--\eqref{tcfSerre3}. Then the extra relation \eqref{eq:add-rel-gr} is equivalent to $[\bh_{i,1},\bh_{i,3}]=0$.
\end{lem}
\begin{proof}
This is immediate from \eqref{hi1bjr} and \eqref{bi1bi2=hi3} that were established without the use of \eqref{eq:add-rel-gr}.
\end{proof}

\begin{lem}
If $i\ne j$, then the relation \eqref{tcbbf} holds:
\beq\label{bibi-shift}
[\bb_{i,r+1},\bb_{j,s}]-[\bb_{i,r},\bb_{j,s+1}]=0.
\eeq
\end{lem}
\begin{proof}
Let $\mathcal X(r,s):=[\bb_{i,r+1},\bb_{j,s}]-[\bb_{i,r},\bb_{j,s+1}]$. We prove by induction on $r+s$ that  $\mathcal X(r,s)=0$. 

The base case $r=s=0$ follows from \eqref{tcbb}. Suppose now $\mathcal X(r,s)=0$. Applying $[\bh_{i,1},\cdot\,]$ to $\mathcal X(r,s)=0$, we find that 
\beq\label{helper6}
2(\alpha_i,\alpha_i)\mathcal X(r+1,s)+2(\alpha_i,\alpha_j)\mathcal X(r,s+1)=0.
\eeq
Applying $[\bh_{j,1},\cdot\,]$ to $\mathcal X(r,s)=0$, we find that 
\beq\label{helper7}
2(\alpha_i,\alpha_j)\mathcal X(r+1,s)+2(\alpha_j,\alpha_j)\mathcal X(r,s+1)=0.
\eeq
Since the coefficient matrix of the linear system \eqref{helper6}--\eqref{helper7} is invertible, it has only trivial solution. Hence $\mathcal X(r+1,s)=\mathcal X(r,s+1)=0$.
\end{proof}

\begin{lem}\label{lem:helper99}
For $i,j\in \I$ and $r\in\bN$ such that $i\ne j$, we have 
\beq\label{hi3bjr}
[\bh_{i,3},\bb_{j,r}]=2(\alpha_i,\alpha_j)\bb_{j,r+3} .
\eeq
\end{lem}
\begin{proof}
Using \eqref{bhdef}, we have
\begin{align*}
-[\bh_{i,3},\bb_{j,r}]&\stackrel{\eqref{bi1bi2=hi3}}{=}\big[[\bb_{i,1},\bb_{i,2}],\bb_{j,r}\big]\\
&\,\,~=\,\,\,\big[[\bb_{i,1},\bb_{j,r}],\bb_{i,2}\big]+\big[\bb_{i,1},[\bb_{i,2},\bb_{j,r}]\big]\\
&\stackrel{\eqref{bibi-shift}}{=}\big[[\bb_{i,0},\bb_{j,r+1}],\bb_{i,2}\big]+\big[\bb_{i,1},[\bb_{i,0},\bb_{j,r+2}]\big]\\
&\stackrel{\eqref{bi0bi2=0}}{=}\big[\bb_{i,0},[\bb_{j,r+1},\bb_{i,2}]\big]+\big[\bb_{i,1},[\bb_{i,0},\bb_{j,r+2}]\big]\\
&\stackrel{\eqref{bibi-shift}}{=}\big[\bb_{i,0},[\bb_{j,r+2},\bb_{i,1}]\big]+\big[\bb_{i,1},[\bb_{i,0},\bb_{j,r+2}]\big]\\
&\,\,~=\,\,\,\big[[\bb_{i,1},\bb_{j,r+2}],\bb_{i,0}\big]+\big[\bb_{i,1},[\bb_{i,0},\bb_{j,r+2}]\big]\\
&\,\,~=\,\,\,\big[[\bb_{i,1},\bb_{i,0}],\bb_{j,r+2}\big]\\
&\,\stackrel{\eqref{bi0bi1=hi1}}{=}-[\bh_{i,1},\bb_{j,r+2}]\stackrel{\eqref{hi1bjr}}{=}-2(\alpha_i,\alpha_j)\bb_{j,r+3}.\qedhere
\end{align*}
\end{proof}

\subsubsection*{\bf Step 2} Next we discuss the relation $[\bh_{i,1},\bh_{i,3}]=0$ which is closely related to  the extra relation \eqref{eq:add-rel-gr}, see Lemma \ref{lem:ext}.

\begin{lem}
Let $i,j\in\I$ be such that $c_{ij}=-1$, then
\beq\label{bi2bi2bj0}
\big[\bb_{i,0},[\bb_{i,0},\bb_{j,s}]\big]=-d_i\bb_{j,s}\quad \text{and}\quad\big[\bb_{i,2},[\bb_{i,2},\bb_{j,0}]\big]=-d_i\bb_{j,4}.
\eeq
\end{lem}
\begin{proof}
Set
\[
\mathcal X(r_1,r_2;s)=\big[\bb_{i,r_1},[\bb_{i,r_2},\bb_{j,s}]\big]+\big[\bb_{i,r_2},[\bb_{i,r_1},\bb_{j,s}]\big].
\]Note that $\mathcal X(r_1,r_2;s)$ is symmetric with respect to $r_1,r_2$.

We show that $\mathcal X(0,0;s)=-2d_i\bb_{j,s}$ by induction on $s$, proving the first equality. Indeed, the induction base comes from \eqref{tcfSerre1}. For the induction step, applying $[\bh_{i,1},\cdot\,]$ and $[\bh_{j,1},\cdot\,]$ to $\mathcal X(0,0;s)=-2d_i\bb_{j,s}$ and then using \eqref{hi1bjr}, we have
\begin{align*}
4(\alpha_i,\alpha_i)\mathcal X(1,0;s)+2(\alpha_i,\alpha_j)\mathcal X(0,0;s+1)&=-4(\alpha_i,\alpha_j)d_i\bb_{j,s+1},\\
4(\alpha_i,\alpha_j)\mathcal X(1,0;s)+2(\alpha_j,\alpha_j)\mathcal X(0,0;s+1)&=-4(\alpha_j,\alpha_j)d_i\bb_{j,s+1}.
\end{align*}
This system has a unique solution given by $\mathcal X(1,0;s)=0$ and $\mathcal X(0,0;s+1)=-2d_i\bb_{j,s+1}$, completing the induction. Thus we have $\big[\bb_{i,0},[\bb_{i,0},\bb_{j,4}]\big]=-d_i\bb_{j,4}$ and hence
\begin{align*}
\big[\bb_{i,2},[\bb_{i,2},\bb_{j,0}]\big]&\stackrel{\eqref{bibi-shift}}{=}\big[\bb_{i,2},[\bb_{i,0},\bb_{j,2}]\big]\\&\stackrel{\eqref{bi0bi2=0}}{=}\big[\bb_{i,0},[\bb_{i,2},\bb_{j,2}]\big]\stackrel{\eqref{bibi-shift}}{=}\big[\bb_{i,0},[\bb_{i,0},\bb_{j,4}]\big]=-d_i\bb_{j,4}.\qedhere
\end{align*}
\end{proof}

\begin{lem}\label{lem:hi4}
If there exist $i,j\in\I$ such that $c_{ij}=c_{ji}=-1$, then we have
\beq\label{hi4}
[\bh_{i,1},\bh_{i,3}]=\bh_{i,4}=[\bb_{i,0},\bb_{i,4}]=[\bb_{i,1},\bb_{i,3}]=0.
\eeq
\end{lem}
\begin{proof}
First, observe that
\begin{align}
[\bh_{i,1},\bh_{i,3}]&\stackrel{\eqref{bi1bi2=hi3}}{=}\big[\bh_{i,1},[\bb_{i,2},\bb_{i,1}]\big]\stackrel{\eqref{hi1bjr}}{=}2(\alpha_i,\alpha_i)[\bb_{i,3},\bb_{i,1}],\label{helper8}\\
[\bh_{i,1},\bh_{i,3}]&\stackrel{\eqref{bhdef}}{=}\big[\bh_{i,1},[\bb_{i,0},\bb_{i,3}]\big]\stackrel{\eqref{hi1bjr}}{=}2(\alpha_i,\alpha_i)\big([\bb_{i,1},\bb_{i,3}]+[\bb_{i,0},\bb_{i,4}]\big).\label{helper9}
\end{align}
Now applying $[\cdot,\bb_{j,0}]$ to the second equality in  \eqref{bi2bi2bj0}, we find that
\beq\label{helper10}
\Big[\bb_{i,2},\big[[\bb_{i,2},\bb_{j,0}],\bb_{j,0}\big]\Big]=-d_i[\bb_{j,4},\bb_{j,0}],
\eeq
which vanishes as $\big[[\bb_{i,2},\bb_{j,0}],\bb_{j,0}\big]=-d_j\bb_{i,2}$ by \eqref{bi2bi2bj0}. Similarly, we have $[\bb_{i,4},\bb_{i,0}]=0$. Now the claim follows by comparing the RHS of \eqref{helper8}--\eqref{helper9}.
% Switching the order and using \eqref{hi3bjr}, we have
% \beq
% [\bh_{i,3},\bh_{i,1}]=\big[\bh_{i,3},[\bb_{i,0},\bb_{i,1}]\big]=2(\alpha_i,\alpha_i)\big([\bb_{i,3},\bb_{i,1}]+[\bb_{i,0},\bb_{i,4}]\big).\label{helper10}
% \eeq
% Comparing \eqref{helper9} and \eqref{helper10}, we have $[\bb_{i,0},\bb_{i,4}]=0$ and hence $\bh_{i,4}=0$ by \eqref{bhdef}. Then it follows from \eqref{helper8}--\eqref{helper9} that $[\bb_{i,1},\bb_{i,3}]=0$. Thus $[\bh_{i,3},\bh_{i,1}]=0$.
\end{proof}
\begin{rem}
Note that we used $c_{ij}=-1$ to get \eqref{bi2bi2bj0} and $c_{ji}=-1$ to show \eqref{helper10} vanishes.
\end{rem}

\begin{lem}\label{itoj}
Suppose $[\bh_{i,1},\bh_{i,3}]=0$ and $(\alpha_i,\alpha_j)\ne 0$, then 
\beq\label{hj4}
[\bh_{j,1},\bh_{j,3}]=\bh_{j,4}=[\bb_{j,0},\bb_{j,4}]=[\bb_{j,1},\bb_{j,3}]=0.
\eeq
\end{lem}
\begin{proof}
We have
\begin{align*}
0&\,\,\,=\,\,[\bh_{i,1},\bh_{i,3}]\stackrel{\eqref{bi1bi2=hi3}}{=}\big[\bh_{i,1},[\bb_{i,2},\bb_{i,1}]\big]=\big[[\bh_{i,1},\bb_{i,2}],\bb_{i,1}\big]+\big[\bb_{i,2},[\bh_{i,1},\bb_{i,1}]\big]\\
&\stackrel{\eqref{hi1bjr}}{=}\frac{(\alpha_i,\alpha_i)}{(\alpha_i,\alpha_j)}\big[[\bh_{j,1},\bb_{i,2}],\bb_{i,1}\big]=\frac{(\alpha_i,\alpha_i)}{(\alpha_i,\alpha_j)}\big[\bh_{j,1},[\bb_{i,2},\bb_{i,1}]\big]=\frac{(\alpha_i,\alpha_i)}{(\alpha_i,\alpha_j)}[\bh_{j,1},\bh_{i,3}].
\end{align*}
Thus $[\bh_{j,1},\bh_{i,3}]=0$. Applying $[\,\cdot\,,\bb_{j,0}]$ to it and then using \eqref{hi1bjr} and \eqref{hi3bjr}, we obtain
\[
(\alpha_j,\alpha_j)[\bb_{j,1},\bh_{i,3}]+(\alpha_i,\alpha_j)[\bh_{j,1},\bb_{j,3}]=0.
\]
Applying $[\bb_{j,0},\,\cdot\,]$ to the above equation, a similar calculation using $[\bh_{j,1},\bh_{i,3}]=0$ shows that
\begin{align*}
-4(\alpha_i,\alpha_j)(\alpha_j,\alpha_j)[\bb_{j,1},\bb_{j,3}]+(\alpha_i,\alpha_j)[\bh_{j,1},\bh_{j,3}]=0.
\end{align*}
Note that \eqref{helper8} implies that $[\bh_{j,1},\bh_{j,3}]=2(\alpha_j,\alpha_j)[\bb_{j,3},\bb_{j,1}]$. Combining this with the above equation implies $[\bh_{j,1},\bh_{j,3}]=0$. The remaining equalities then follow as in Lemma \ref{lem:hi4}.
\end{proof}

If $\g$ is not of type $\mathsf A_1$, $\mathsf B_2\cong\mathsf C_2$, or $\mathsf G_2$, then we can always find $i,j\in \I$ such that $c_{ij}=c_{ji}=-1$. Hence there must exist $i\in \I$ such that \eqref{hi4} holds. If $\g$ is of type $\mathsf A_1$, $\mathsf B_2\cong\mathsf C_2$, or $\mathsf G_2$, then we manually include the extra relation \eqref{eq:add-rel-gr} for a single $i\in \I$ which implies \eqref{hi4} by Lemma \ref{lem:ext}. Hence we always have the relation \eqref{hi4} for at least a single $i\in \I$ in the algebra $\mathcal U$. Then Lemma \ref{itoj} implies further that \eqref{hi4} holds for all $i\in \I$.

\begin{lem}
If $[\bh_{i,1},\bh_{i,3}]=0$, then
\beq
[\bh_{i,3},\bb_{i,r}]=2(\alpha_i,\alpha_i)\bb_{i,r+3}.
\eeq
\end{lem}
\begin{proof}
Since $[\bh_{i,1},\bh_{i,3}]=0$, it suffices to prove it for the case $r=0$ and the general case follows by recursively applying $[\bh_{i,1},\cdot\,]$. Indeed, we have
\begin{align*}
[\bh_{i,3},\bb_{i,0}]&\stackrel{\eqref{bi1bi2=hi3}}{=}\big[[\bb_{i,2},\bb_{i,1}],\bb_{i,0}\big]\\
&\stackrel{\eqref{bi0bi2=0}}{=}\big[\bb_{i,2},[\bb_{i,1},\bb_{i,0}]\big]\stackrel{\eqref{hi1bjr}}{=}[\bh_{i,1},\bb_{i,2}]\stackrel{\eqref{bhdef}}{=}2(\alpha_i,\alpha_i)\bb_{i,3}.\qedhere
\end{align*}
\end{proof}

\subsubsection*{\bf Step 3} Finally, with the previous lemmas, we are ready to finish the proof of Proposition \ref{prop:red}. 

We first prove by induction on $\ell$ that
\begin{align}
&[\bh_{i,r},\bh_{i,s}]=0,\quad \bh_{i,2p}=0,\label{todo1}\\
&[\bh_{i,r},\bb_{i,s}]=2(\alpha_i,\alpha_i)\bb_{i,r+s},\label{todo2}\\
&[\bb_{i,r},\bb_{i,s}]=(-1)^r\bh_{i,r+s},\label{todo3}
\end{align}
if $r+s\lle\ell$ and $2p\lle \ell$. This verifies  \eqref{tchhf}--\eqref{tcbbf} for the case $i=j$.

We have already established these equalities for $\ell\lle 4$, see \eqref{bhdef}--\eqref{bi1bi2=hi3} and Lemmas \ref{lem:helper99}, \ref{lem:hi4}, \ref{itoj}. We shall proceed with a case-by-case study depending on the remainder of $\ell$ divided by $4$.

If $\ell=4l+1$ for $l\gge 1$, then there are no new relations from \eqref{todo1}. By induction hypothesis we have 
\beq\label{helper11}
\bh_{i,4l}=[\bb_{i,0},\bb_{i,4l}]=\dots=[\bb_{i,2l-1},\bb_{i,2l+1}]=0.
\eeq
Applying $[\bh_{i,1},\cdot\,]$ to the above equation, we have
\begin{align*}
0&=[\bb_{i,0},\bb_{i,4l+1}]+[\bb_{i,1},\bb_{i,4l}]\\&=[\bb_{i,1},\bb_{i,4l}]+[\bb_{i,2},\bb_{i,4l-1}]=\cdots=[\bb_{i,2l-1},\bb_{i,2l+2}]+[\bb_{i,2l},\bb_{i,2l+1}].
\end{align*}
Together with \eqref{bhdef}, it implies \eqref{todo3} for $r+s=4l+1$. 

Observe that if $[\bh_{i,2r+1},\bh_{i,1}]=0$, then $[\bh_{i,2r+1},\bb_{i,0}]=2(\alpha_i,\alpha_i)\bb_{i,2r+1}$ implies $[\bh_{i,2r+1},\bb_{i,s}]=2(\alpha_i,\alpha_i)\bb_{i,2r+s}$ for $s\in\bN$ by recursively applying $[\bh_{i,1},\cdot\,]$. Thus to show \eqref{todo2}, it suffices to prove $[\bh_{i,4l+1},\bb_{i,0}]=2(\alpha_i,\alpha_i)\bb_{i,4l+1}$. This can be proved as follows:
\beq\label{helper12}
\begin{split}
[\bh_{i,4l+1},\bb_{i,0}]&\stackrel{\eqref{todo3}}{=}\big[\bb_{i,0},[\bb_{i,1},\bb_{i,4l}]\big]\\
&\stackrel{\eqref{helper11}}{=}\big[[\bb_{i,0},\bb_{i,1}],\bb_{i,4l}\big]\stackrel{\eqref{bhdef}}{=}[\bh_{i,1},\bb_{i,4l}]\stackrel{\eqref{hi1bjr}}{=}2(\alpha_i,\alpha_i)\bb_{i,4l+1}.
\end{split}
\eeq

If $\ell=4l+2$, then $[\bh_{i,2l+1},\bh_{i,1}]=0$ by induction hypothesis and hence we have $[\bh_{i,2l+1},\bb_{i,s}]=2(\alpha_i,\alpha_i)\bb_{i,2l+s+1}$ for all $s\in\bN$. Thus
\beq\label{bi0bi4l+2}
0=[\bh_{i,2l+1},\bh_{i,2l+1}]\stackrel{\eqref{bhdef}}{=}\big[\bh_{i,2l+1},[\bb_{i,0},\bb_{i,2l+1}]\big]=2(\alpha_i,\alpha_i)[\bb_{i,0},\bb_{i,4l+2}].
\eeq
Note that for $0\lle r\lle 2l$ we have
\begin{align*}
[\bh_{i,1},\bh_{i,4l+1}]&=(-1)^r\big[\bh_{i,1},[\bb_{i,r},\bb_{i,4l+1-r}]\big]\\
&=2(\alpha_i,\alpha_i)(-1)^r\big([\bb_{i,r},\bb_{i,4l+2-r}]+[\bb_{i,r+1},\bb_{i,4l+1-r}]\big).
\end{align*}
Thus
\begin{align*}
\tfrac{1}{2(\alpha_i,\alpha_i)}[\bh_{i,1},\bh_{i,4l+1}]&=[\bb_{i,0},\bb_{i,4l+2}]+[\bb_{i,1},\bb_{i,4l+1}]=-[\bb_{i,1},\bb_{i,4l+1}]-[\bb_{i,2},\bb_{i,4l}]\\&=\cdots=-[\bb_{i,2l-1},\bb_{i,2l+3}]-[\bb_{i,2l},\bb_{i,2l+2}]=[\bb_{i,2l},\bb_{i,2l+2}].
\end{align*}
Set $\mathfrak X:=\tfrac{1}{2(\alpha_i,\alpha_i)}[\bh_{i,1},\bh_{i,4l+1}]$, then it is easy to see from the equation above that 
$$
[\bb_{i,2l+1-r},\bb_{i,2l+1+r}]=(-1)^{r+1}r\mathfrak X.
$$ 
Now setting $r=2l+1$, we obtain from \eqref{bi0bi4l+2} that $\mathfrak X=0$. In particular, we obtain \eqref{todo3} for $r+s=4l+2$, and 
\[
\bh_{i,4l+2}\stackrel{\eqref{bhdef}}{=}0, \qquad [\bh_{i,1},\bh_{i,4l+1}]=0.
\]
The relation \eqref{todo2} for $r+s=4l+2$ is clear by induction hypothesis, \eqref{helper12} and  $[\bh_{i,1},\bh_{i,r}]=0$ for all $1\lle r\lle 4l+1$. Actually, we also have $[\bh_{i,4l+1},\bb_{i,s}]=2(\alpha_i,\alpha_i)\bb_{i,4l+s+1}$ for all $s\in\bN$. Finally, by rewriting $\bh_{i,r}$ as $(-1)^{r-1}[\bb_{i,r-1},\bb_{i,1}]$, the relation \eqref{todo1} for $r+s=4l+2$ follows from \eqref{todo3} for $r+s=4l+2$.

If $\ell=4l+3$, then it is very similar to the case $\ell=4l+1$. Hence we omit the details.

If $\ell=4l+4$, set $\mathfrak X:=\tfrac{1}{2(\alpha_i,\alpha_i)}[\bh_{i,1},\bh_{i,4l+3}]$. As in the case $\ell=4l+2$, one finds that 
\beq\label{helper13}
[\bb_{i,2l+2-r},\bb_{i,2l+2+r}]=(-1)^{r}r\mathfrak X.
\eeq
On one hand,
\begin{align*}
[\bh_{i,2l+1},\bh_{i,2l+3}]&\stackrel{\eqref{todo3}}{=}\big[\bh_{i,2l+1},[\bb_{i,2l},\bb_{i,3}]\big]\\
&\stackrel{\eqref{todo2}}{=}2(\alpha_i,\alpha_i)\big([\bb_{i,4l+1},\bb_{i,3}]+[\bb_{i,2l},\bb_{i,2l+4}]\big)\\
&\stackrel{\eqref{helper13}}{=}2(\alpha_i,\alpha_i)\big((2l-1)\mathfrak X+2\mathfrak X\big)=2(\alpha_i,\alpha_i)(2l+1)\mathfrak X.
\end{align*}
On the other hand,
\begin{align*}
[\bh_{i,2l+1},\bh_{i,2l+3}]&\stackrel{\eqref{todo3}}{=}\big[\bh_{i,2l+3},[\bb_{i,1},\bb_{i,2l}]\big]\\
&\stackrel{\eqref{todo2}}{=}2(\alpha_i,\alpha_i)\big([\bb_{i,2l+4},\bb_{i,2l}]+[\bb_{i,1},\bb_{i,4l+3}]\big)\\
&\stackrel{\eqref{helper13}}{=}-2(\alpha_i,\alpha_i)\big(2\mathfrak X+(2l+1)\mathfrak X\big)=-2(\alpha_i,\alpha_i)(2l+3)\mathfrak X.
\end{align*}
Thus we conclude that $\mathfrak X=0$. In particular, $\bh_{i,4l+4}=[\bb_{i,0},\bb_{i,4l+4}]=0$. The rest is very similar to the case $\ell=4l+2$. Therefore, we have proved \eqref{todo1}--\eqref{todo3}.

Due to \eqref{todo1}--\eqref{todo3} and \eqref{bibi-shift}, to prove Proposition \ref{prop:red}, it remains to prove \eqref{tchhf} and \eqref{tchbf} for $i\ne j$. 

If $i\ne j$, then we have
\begin{align*}
[\bh_{i,2r+1},\bb_{j,s}]&\stackrel{\eqref{todo3}}{=}\big[[\bb_{i,2r},\bb_{i,1}],\bb_{j,s}\big]=\big[[\bb_{i,2r},\bb_{j,s}],\bb_{i,1}\big]+\big[\bb_{i,2r},[\bb_{i,1},\bb_{j,s}]\big]\\
&\stackrel{\eqref{bibi-shift}}{=}\big[[\bb_{i,0},\bb_{j,2r+s}],\bb_{i,1}\big]+\big[\bb_{i,2r},[\bb_{i,0},\bb_{j,s+1}]\big]\\
&\overset{\eqref{todo1}}{\underset{\eqref{todo3}}{=}}\big[[\bb_{i,0},\bb_{j,2r+s}],\bb_{i,1}\big]+\big[\bb_{i,0},[\bb_{i,2r},\bb_{j,s+1}]\big]\\
&\stackrel{\eqref{bibi-shift}}{=}\big[[\bb_{i,0},\bb_{j,2r+s}],\bb_{i,1}\big]+\big[\bb_{i,0},[\bb_{i,1},\bb_{j,2r+s}]\big]\\
&\,\,\,=\,\,\,\big[[\bb_{i,0},\bb_{i,1}],\bb_{j,2r+s}\big]=[\bh_{i,1},\bb_{j,2r+s}]=2(\alpha_i,\alpha_j)\bb_{j,2r+s+1},
\end{align*}
proving \eqref{tchbf}. Then applying $[\,\cdot\,,\bb_{j,0}]$ to $[\bh_{i,2r+1},\bb_{j,2s+1}]=2(\alpha_i,\alpha_j)\bb_{j,2r+2s+2}$, we obtain
\begin{align*}
2(\alpha_i,\alpha_j)[\bb_{j,2r+1},\bb_{j,2s+1}]-[\bh_{i,2r+1},\bh_{j,2s+1}]=2(\alpha_i,\alpha_j)[\bb_{j,2r+2s+2},\bb_{j,0}].
\end{align*}
Note that by \eqref{todo1} and \eqref{todo3}, we have $[\bb_{j,0},\bb_{j,2r+2s+2}]=[\bb_{j,2r+1},\bb_{j,2s+1}]=0$. Together with the equation above, we conclude that $[\bh_{i,2r+1},\bh_{j,2s+1}]=0$, completing the proof of \eqref{tchhf}.

Combining all the observations together, the proof of Proposition \ref{prop:red} is complete.

\subsection{Proof of Theorem \ref{thm:min-text}}\label{sec:pfthmA}
Let $\Yi^{\min}$ be the algebra generated by $h_{i,1},b_{i,0},b_{i,1}$ for $i\in \I$ subject to the relations \eqref{redhh}--\eqref{redbb}, the finite Serre type relations \eqref{fSerre0}--\eqref{fSerre3}, and the extra relation \eqref{eq:add-rel} when $\g$ is of type $\mathsf A_1$, $\mathsf B_2\cong \mathsf C_2$, or $\mathsf G_2$. Then we have an algebra homomorphism
\begin{align*}
\pi:&~ \Yi^{\min}\to \Yi\\
&~h_{i,1}\mapsto h_{i,1},\quad b_{i,r}\mapsto b_{i,r} \qquad (i\in \I,r=0,1).
\end{align*}
Note that when the extra relation \eqref{eq:add-rel} is included, then it is also satisfied in $\Yi$ by Lemma \ref{lem:hbhb}.
Moreover, $\pi$ is surjective as $\Yi$ is generated by $h_{i,1},b_{i,0}$ for $i\in \I$ by Lemma \ref{lem:generate}. 

Introduce a filtration on $\Yi^{\min}$  by setting $\deg h_{i,1}=1$ and $\deg b_{i,r}=r$. Recall the filtration on $\Yi$ from Section \ref{sec:ty}. Then $\pi$ is clearly a filtered homomorphism. Taking the associated graded, we obtain a surjective homomorphism $\gr\,\pi:\gr\,\Yi^{\min}\to\gr\,\Yi$.

On the other hand, recall the algebra isomorphism $\varrho: \mathrm{U}(\g[z]^{\check\omega})\stackrel{\cong\,}{\longrightarrow} \mathrm{gr}\,\Yi$ from \eqref{assi}. By the presentation of $\mathrm{U}(\g[z]^{\check\omega})$ from Proposition \ref{prop:red}, we also have a natural homomorphism $\mathrm{U}(\g[z]^{\check\omega})\to \gr\,\Yi^{\min}$, which matches on generators (in different math fonts). By composing  $\varrho^{-1}$ with this map, we obtain a homomorphism $\gr \,\Yi\to \gr\,\Yi^{\min}$; by \eqref{assi}, \eqref{eq:hb-explicit} and \eqref{eq:psi}, this homomorphism is clearly the inverse to $\gr\,\pi$ as seen on generators. Hence, we have proved that $\gr\,\pi:\gr\,\Yi^{\min}\to\gr\,\Yi$ is an isomorphism, and so is $\pi:\Yi^{\min}\to \Yi$.

\section{Coideal structures}\label{sec:coideal}
\subsection{Embedding into Yangians}
Our next main result identifies the twisted Yangian $\Yi$ in the Drinfeld presentation with the twisted Yangian $\YiJ$ in the $J$ presentation.
\begin{thm}\label{thm:embedmaintext}
Let $\g$ be a simple Lie algebra which is not of type $\mathsf G_2$. We have the following.
\begin{enumerate}
    \item The map $\varphi:\Yi\to \Y$ defined by the rule 
\begin{align}
    h_{i,1}&\mapsto 2\xi_{i,1}-\xi_{i,0}^2+\sum_{\alpha\in\Phi^+}(\alpha,\alpha_i)(x_\alpha^+)^2,\label{hi1-embedding}\\
    b_{i,0}&\mapsto x_{i}^+-x_{i}^-,\label{bi0-embedding}\\
    b_{i,1}&\mapsto x_{i,1}^++x_{i,1}^-+\tfrac12\sum_{\alpha\in\Phi^+}\big\{[x_{i}^+,x_{\alpha}^+],x_{\alpha}^+\big\}-\tfrac12\{x_i^+,\xi_i\},\label{bi1-embedding}
\end{align}
induces an injective algebra homomorphism from $\Yi$ to $\Y$.
\item Identifying $\Yi$ as a subalgebra of $\Y$ via $\varphi$,  we have
\beq\label{coprohi1}
\Delta(h_{i,1})=h_{i,1}\otimes 1+1\otimes h_{i,1}+2\sum_{\alpha\in\Phi^+}(\alpha,\alpha_i)(x_{\alpha}^+-x_{\alpha}^-)\otimes x_\alpha^+.
\eeq
In particular, $\Yi$ is a right coideal subalgebra of $\Y$, i.e. $\Delta(\Yi)\subset \Yi\otimes \Y$.
\item The twisted Yangian $\Yi$ in the Drinfeld presentation is isomorphic to the twisted Yangian $\YiJ$ in the $J$ presentation. 
\end{enumerate}
\end{thm}
The theorem will be proved in Sections \ref{sec:pfthmB0}--\ref{sec:pfthmB} with the help of Theorem \ref{thm:min-text}. For the case of type $\mathsf A_1$, it will be treated differently in Appendix \ref{sec:app} via the isomorphism between twisted Yangians in their R-matrix and Drinfeld presentations, established in \cite{LWZ25GD}. Similarly in Appendix \ref{sec:app2}, the case of type $\mathsf B_2\cong \mathsf C_2$ requires extra help from twisted Yangians in their R-matrix presentation \cite{GR16} to deal with the extra relation \eqref{eq:add-rel}.

Note that the formulas in \eqref{hi1-embedding} and \eqref{bi1-embedding} are surprisingly close to the identification formulas between Yangians in their $J$  and Drinfeld presentations from \eqref{eq:J}. In terms of the $J$ presentation, the elements $h_{i,1}$ corresponds to
\beq\label{hi1-J}
\begin{split}
h_{i,1}\mapsto &2J(\xi_i)-\tfrac12[\xi_i,C_{\mathfrak k}]+\tfrac{1}2\sum_{\alpha\in\Phi^+}(\alpha,\alpha_i)(x_{\alpha}^+-x_{\alpha}^-)^2.
\end{split}
\eeq
In terms of $B(\xi_i)$ from Definition \ref{def:tY-J},  
this can also be written as 
\[
h_{i,1}\mapsto 2B(\xi_i)+\tfrac{1}2\sum_{\alpha\in\Phi^+}(\alpha,\alpha_i)(b_{\alpha})^2.
\]

The coproduct for $b_{i,1}$ can be explicitly computed using the relation
\[
\Delta(b_{i,1})=\frac{1}{2(\alpha_i,\alpha_i)}\big[\Delta(h_{i,1}),b_{i,0}\otimes 1+1\otimes b_{i,0}\big]
\]
together with similar calculations to those performed in \S\ref{sec:pfthmB0}. The end result is
\beq\label{helper98}
\begin{split}
\Delta(b_{i,1})=\,&b_{i,1}\otimes 1+1\otimes b_{i,1}-b_{i,0}\otimes \xi_i\\&-\sum_{\alpha\in\Phi^+}\Big(\big([x_\alpha^+,x_i^+]-[x_i^-,x_\alpha^-]\big)\otimes x_\alpha^++(x_\alpha^+-x_\alpha^-)\otimes [x_\alpha^+,x_i^+]\Big).
\end{split}
\eeq

\subsection{$\varphi$ induces an algebra homomorphism}\label{sec:pfthmB0}
In this subsection, we prove that $\varphi$ from Part (1) of Theorem \ref{thm:embedmaintext} induces an algebra homomorphism with the help of Theorem \ref{thm:min-text}. We shall verify below that $\varphi(h_{i,1}),\varphi(b_{i,0}),\varphi(b_{i,1})$ satisfy the relations \eqref{redhh}--\eqref{redbb} (the relations \eqref{fSerre0}--\eqref{fSerre3} are straightforward and well known) and hence $\varphi$ induces an algebra homomorphism if $\g$ is not of type $\mathsf A_1$, $\mathsf B_2\cong \mathsf C_2$, or $\mathsf G_2$, which we denote again by $\varphi$. 

For the case of type $\mathsf A_1$, it is already known in \cite{LWZ25GD} that $\Yi$ is a subalgebra of $\Y$ as $\Yi$ is isomorphic to the (special) twisted Yangian of type $\mathsf{AI}$ in the R-matrix presentation \cite{Ols92,MNO96} which is known to be a coideal subalgebra of the Yangian for $\fksl_2$. Hence we only need to carefully identify the elements $h_{i,1}$ and $b_{i,1}$ in terms of Drinfeld generators in $\Y$. This is accomplished in Appendix \ref{sec:app}. For the case of type $\mathsf B_2\cong \mathsf C_2$, we need to verify the extra relation \eqref{eq:add-rel}, this is done in Appendix \ref{sec:app2} via the Gauss decomposition of twisted Yangians in their R-matrix presentation from \cite{GR16}. 

In the next subsection, we show the algebra homomorphism $\varphi$ is injective by passing to the associated graded algebras. Finally, we prove Parts (2) and (3) with the help of Part (1).

\subsubsection{The relation \eqref{redhb}} 
We need to verify
\beq\label{HBrel}
[\varphi(h_{i,1}),\varphi(b_{j,0})]=2(\alpha_i,\alpha_j)\varphi(b_{j,1}).
\eeq
Applying \eqref{tlxidef}, \eqref{hi1-embedding} and \eqref{bi0-embedding}, the LHS of \eqref{HBrel} is given by
\begin{align*}
&\Big[2\tl\xi_{i,1}+\sum_{\alpha\in\Phi^+}(\alpha,\alpha_i)(x_{\alpha}^+)^2,x_j^+-x_j^-\Big ]\\
\stackrel{\eqref{tlxicom}}{=}\,\, &2(\alpha_i,\alpha_j)(x_{j,1}^++x_{j,1}^-)+\sum_{\alpha\in\Phi^+}(\alpha,\alpha_i)\Big(\big\{x_{\alpha}^+,[x_\alpha^+,x_j^+]\big\}-\big\{x_{\alpha}^+,[x_\alpha^+,x_j^-]\big\}\Big)\\
\stackrel{(*)}{=}\,\,~\,&2(\alpha_i,\alpha_j)(x_{j,1}^++x_{j,1}^-)+\sum_{\alpha\in\Phi^+\setminus\{\alpha_j\}}(\alpha,\alpha_i) \big\{x_{\alpha}^+,[x_\alpha^+,x_j^+]\big\}\\
& -(\alpha_i,\alpha_j)\{\xi_{j},x_j^+\}-\sum_{\alpha\in\Phi^+\setminus\{\alpha_j\}}(\alpha+\alpha_j,\alpha_i) \big\{x_{\alpha+\alpha_j}^+,[x_{\alpha+\alpha_j}^+,x_j^-]\big\}\\
\stackrel{\eqref{+++=+-+}}{=}\,\,& 2(\alpha_i,\alpha_j)(x_{j,1}^++x_{j,1}^-)-(\alpha_i,\alpha_j)\{\xi_{j},x_j^+\}-\sum_{\alpha\in\Phi^+\setminus\{\alpha_j\}}(\alpha_j,\alpha_i) \big\{x_{\alpha}^+,[x_\alpha^+,x_j^+]\big\}\\
\stackrel{\eqref{bi1-embedding}}{=}\, \,& 2(\alpha_i,\alpha_j)\varphi(b_{j,1}),
\end{align*}
where we have used \eqref{eq:relXX} and applied substitution $\alpha\mapsto\alpha+\alpha_j$ in $(*)$ and the equalities
\beq\label{+++=+-+}
x_{\alpha}^+[x_\alpha^+,x_j^+]\stackrel{\eqref{eta}}{=}[x_{\alpha+\alpha_j}^+,x_j^-]x_{\alpha+\alpha_j}^+,\qquad [x_\alpha^+,x_j^+]x_{\alpha}^+\stackrel{\eqref{eta}}{=}x_{\alpha+\alpha_j}^+[x_{\alpha+\alpha_j}^+,x_j^-].
\eeq
Note that throughout this section, we assume that $x_\alpha^\pm=0$ if $\alpha\notin \Phi^+$.

\subsubsection{The relation \eqref{redbb}}
We need to verify
\begin{align}
&\big[\varphi(b_{i,0}),\varphi(b_{i,1})\big]= \varphi(h_{i,1})-\tfrac12(\alpha_i,\alpha_i)\big(\varphi(b_{i,0})\big)^2,\label{bbi=j}\\
&\big[\varphi(b_{i,1}),\varphi(b_{j,0})\big]-\big[\varphi(b_{i,0}),\varphi(b_{j,1})\big]= \tfrac12(\alpha_i,\alpha_j)\big\{\varphi(b_{i,0}),\varphi(b_{j,0})\big\},	\quad (i\ne j).\label{bbinej}
\end{align}
We start with the following lemma.
\begin{lem}
For $i\ne j$, we have
\begin{align}
&\Big[x_i^+-x_i^-,\sum_{\alpha\in\Phi^+}\big\{[x_i^+,x_\alpha^+],x_{\alpha}^+\big\}\Big]=2\sum_{\alpha\in\Phi^+\setminus\{\alpha_i\}}(\alpha,\alpha_i)(x_\alpha^+)^2,\label{helper1}\\
&\Big[x_i^+-x_i^-,\sum_{\alpha\in\Phi^+}\big\{[x_j^+,x_\alpha^+],x_{\alpha}^+\big\}\Big]+\Big[x_j^+-x_j^-,\sum_{\alpha\in\Phi^+}\big\{[x_i^+,x_\alpha^+],x_{\alpha}^+\big\}\Big]\label{helper2}\\
&=\big\{[x_i^+,\xi_j],x_j^+\big\}+\big\{[x_j^+,\xi_i],x_i^+\big\}+\big\{[x_i^+,x_j^+],\xi_j\big\}+\big\{[x_j^+,x_i^+],\xi_i\big\}.\notag
\end{align}
\end{lem}
\begin{proof}
The LHS of \eqref{helper1} is equal to
\begin{align*}
&\stackrel{\eqref{eq:relXX}}{=}\sum_{\alpha\in\Phi^+\setminus\{\alpha_i\}}\Big(\big\{\big[x_i^+,[x_i^+,x_\alpha^+]\big],x_\alpha^+\big\}+\big\{[x_i^+,x_\alpha^+],[x_i^+,x_\alpha^+]\big\}\\
&\hskip 2cm+\big\{ [\xi_i,x_\alpha^+] ,x_\alpha^+\big\}-\big\{\big[x_i^+,[x_i^-,x_\alpha^+]\big],x_\alpha^+\big\}-\big\{[x_i^+,x_\alpha^+],[x_i^-,x_\alpha^+]\big\}\Big)\\
&~=\sum_{\alpha\in\Phi^+\setminus\{\alpha_i\}}\Big(\big\{\big[x_i^+,[x_i^+,x_\alpha^+]\big],x_\alpha^+\big\}-\big\{[x_i^+,x_{\alpha+\alpha_i}^+],[x_i^-,x_{\alpha+\alpha_i}^+]\big\}\\
&\hskip 1cm+\big\{ [\xi_i,x_\alpha^+] ,x_\alpha^+\big\}+\big\{[x_i^+,x_\alpha^+],[x_i^+,x_\alpha^+]\big\}-\big\{\big[x_i^+,[x_i^-,x_{\alpha+\alpha_i}^+]\big],x_{\alpha+\alpha_i}^+\big\}\Big)\\
&\stackrel{\eqref{eta}}{=}\sum_{\alpha\in\Phi^+\setminus\{\alpha_i\}} \big\{ [\xi_i,x_\alpha^+] ,x_\alpha^+\big\}\stackrel{\eqref{eq:relHX}}{=} 2\sum_{\alpha\in\Phi^+\setminus\{\alpha_i\}}(\alpha,\alpha_i)(x_\alpha^+)^2,
\end{align*}
completing the proof of \eqref{helper1}.

The LHS of \eqref{helper2} is equal to
\begin{align*}
&\sum_{\alpha\in\Phi^+\setminus\{\alpha_j\}}\bigg(\Big(\big\{\big[x_i^+,[x_j^+,x_\alpha^+]\big],x_\alpha^+\big\}+\big\{[x_j^+,x_\alpha^+],[x_i^+,x_\alpha^+]\big\}\Big)\\
&\qquad\qquad\qquad-\Big(\big\{\big[x_j^+,[x_i^-,x_\alpha^+]\big],x_\alpha^+\big\}+\big\{[x_j^+,x_\alpha^+],[x_i^-,x_\alpha^+]\big\}\Big)\bigg)\\
+&\sum_{\alpha\in\Phi^+\setminus\{\alpha_i\}}\bigg(\Big(\big\{\big[x_j^+,[x_i^+,x_\alpha^+]\big],x_\alpha^+\big\}+\big\{[x_i^+,x_\alpha^+],[x_j^+,x_\alpha^+]\big\}\Big)\\
&\qquad\qquad\qquad-\Big(\big\{\big[x_i^+,[x_j^-,x_\alpha^+]\big],x_\alpha^+\big\}+\big\{[x_i^+,x_\alpha^+],[x_j^-,x_\alpha^+]\big\}\Big)\bigg)\\
\stackrel{(*)}{=}&\sum_{\alpha\in\Phi^+\setminus\{\alpha_i,\alpha_j\}}\bigg(\Big(\big\{\big[x_i^+,[x_j^+,x_\alpha^+]\big],x_\alpha^+\big\}-\big\{[x_i^+,x_{\alpha+\alpha_j}^+],[x_j^-,x_{\alpha+\alpha_j}^+]\big\}\Big)\\
&\qquad\qquad\qquad-\Big(\big\{\big[x_j^+,[x_i^-,x_{\alpha+\alpha_i}^+]\big],x_{\alpha+\alpha_i}^+\big\}-\big\{[x_i^+,x_\alpha^+],[x_j^+,x_\alpha^+]\big\}\Big)\bigg)\\
+&\sum_{\alpha\in\Phi^+\setminus\{\alpha_i,\alpha_j\}}\bigg(\Big(\big\{\big[x_j^+,[x_i^+,x_\alpha^+]\big],x_\alpha^+\big\}-\big\{[x_j^+,x_{\alpha+\alpha_i}^+],[x_i^-,x_{\alpha+\alpha_i}^+]\big\}\Big)\\
&\qquad\qquad\qquad-\Big(\big\{\big[x_i^+,[x_j^-,x_{\alpha+\alpha_j}^+]\big],x_{\alpha+\alpha_j}^+\big\}-\big\{[x_j^+,x_\alpha^+],[x_i^+,x_\alpha^+]\big\}\Big)\bigg)\\
+&\,\big\{[x_i^+,\xi_j],x_j^+\big\}+\big\{[x_j^+,\xi_i],x_i^+\big\}+\big\{[x_i^+,x_j^+],\xi_j\big\}+\big\{[x_j^+,x_i^+],\xi_i\big\}\\
\stackrel{\eqref{eta}}{=}&\,\{[x_i^+,\xi_j],x_j^+\big\}+\big\{[x_j^+,\xi_i],x_i^+\big\}+\big\{[x_i^+,x_j^+],\xi_j\big\}+\big\{[x_j^+,x_i^+],\xi_i\big\}.
\end{align*}
In $(*)$ we applied \eqref{eq:relXX} and suitable substitutions for $\alpha$ in certain terms, and then arranged the terms properly so that the terms in parenthesis eventually cancel.
\end{proof}

We are ready to check the relation \eqref{bbi=j}. By \eqref{bi0-embedding}--\eqref{bi1-embedding}, the LHS of \eqref{bbi=j}  is equal to 
\begin{align*}
&\Big[x_i^+-x_i^-,x_{i,1}^++x_{i,1}^-+\tfrac12\sum_{\alpha\in\Phi^+}\big\{\big[x_i^+,x_{\alpha}^+\big],x_{\alpha}^+\big\}-\tfrac12\big\{x_i^+,\xi_i\big\}\Big]\\
\stackrel{\eqref{eq:relXX}}{=}\,&[x_i^+,x_{i,1}^+]+2\xi_{i,1}-[x_i^-,x_{i,1}^-]-\tfrac12\big\{x_i^+,[x_i^+,\xi_i]\big\}-\tfrac12\{\xi_i,\xi_i\}\\
&\hskip1cm +\tfrac12\big\{x_i^+,[x_i^-,\xi_i]\big\}+\hf\Big[x_i^+-x_i^-,\sum_{\alpha\in\Phi^+}\big\{[x_i^+,x_\alpha^+],x_{\alpha}^+\big\}\Big]\\
\overset{\eqref{eq:relHX}}{\underset{\eqref{eq:relexXX}}{=}}\,&
-\tfrac12(\alpha_i,\alpha_i)(x_i^+)^2+2\xi_{i,1}-\tfrac12(\alpha_i,\alpha_i)(x_i^-)^2+(\alpha_i,\alpha_i)(x_i^+)^2-\xi_i^2\\
&\hskip1cm+\tfrac12(\alpha_i,\alpha_i)\{x_i^+,x_i^-\}+\hf\Big[x_i^+-x_i^-,\sum_{\alpha\in\Phi^+}\big\{[x_i^+,x_\alpha^+],x_{\alpha}^+\big\}\Big]\\
\overset{\eqref{bi0-embedding}}{\underset{\eqref{helper1}}{=}}\,&\,2\xi_{i,1}-\tfrac12(\alpha_i,\alpha_i)\varphi(b_{i,0})^2-\xi_i^2+\sum_{\alpha\in\Phi^+}(\alpha,\alpha_i)(x_\alpha^+)^2\\\stackrel{\eqref{hi1-embedding}}{=}~&\varphi(h_{i,1})-\tfrac12(\alpha_i,\alpha_i)\varphi(b_{i,0})^2,
\end{align*}
which is exactly the RHS of \eqref{bbi=j}.

The verification of \eqref{bbinej} is similar with the help of \eqref{eq:relexXX} and \eqref{helper2}. We omit the details.

\subsubsection{The relation \eqref{redhh}}
Let us verify the more complicated relation \eqref{redhh}, i.e.
\beq\label{helper3}
\big[\varphi(h_{i,1}),\varphi(h_{j,1})\big]=0.
\eeq
For that purpose, let
\[
\tl v_i:=\tfrac14\sum_{\alpha\in\Phi^+}(\alpha,\alpha_i)\big\{x_\alpha^+,x_{\alpha}^-\big\}.
\]
Then it follows from \eqref{tlxidef}, \eqref{eq:J}, and \eqref{hi1-embedding} that
\[
J(\xi_i)=\tl \xi_{i,1}+\tl v_i,\qquad \varphi(h_{i,1})=2J(\xi_i)-2\tl v_i+\sum_{\alpha\in\Phi^+}(\alpha,\alpha_i)(x_\alpha^+)^2
\]
It is known from \cite[Theorem~2.6]{GRW19} that
\[
[J(\xi_i)-\tl v_i,J(\xi_j)-\tl v_j]=0.
\]
Thus the LHS of \eqref{helper3} is reduced to
\beq\label{helper4}
\begin{split}
2\sum_{\beta\in\Phi^+}(\beta,\alpha_j)\big\{[J(\xi_i),x_\beta^+],x_\beta^+\big\}&-2\sum_{\beta\in\Phi^+}(\beta,\alpha_j)\big\{[\tl v_i,x_\beta^+],x_\beta^+\big\}\\
-2\sum_{\alpha\in\Phi^+}(\alpha,\alpha_i)\big\{[J(\xi_j),x_\alpha^+],x_\alpha^+\big\}&+2\sum_{\alpha\in\Phi^+}(\alpha,\alpha_i)\big\{[\tl v_j,x_\alpha^+],x_\alpha^+\big\}\\
&+\sum_{\alpha,\beta\in\Phi^+}(\alpha,\alpha_i)(\beta,\alpha_j)\big[(x_\alpha^+)^2,(x_\beta^+)^2\big].
\end{split}
\eeq
By \cite[Proposition~3.21]{GNW18}, we have
\[
[J(\xi_i),x_\beta^+]=(\beta,\alpha_i)J(x_\beta^+),\qquad [J(\xi_j),x_\alpha^+]=(\alpha,\alpha_j)J(x_\alpha^+).
\]
Thus the two summations in \eqref{helper4} involving $J(\xi_i)$ or $J(\xi_j)$  cancel. Further notice 
\begin{align*}
[\tl v_i,x_\beta^+]&=\Big[\tfrac18 c_{\g} \xi_i+\tfrac12\sum_{\alpha\in\Phi^+}(\alpha,\alpha_i)x_\alpha^-x_\alpha^+,x_\beta^+\Big]\\
&=\tfrac18 c_{\g}(\beta,\alpha_i)x_\beta^++\tfrac12\sum_{\alpha\in\Phi^+}(\alpha,\alpha_i)\big([x_\alpha^-,x_\beta^+]x_\alpha^+ + x_\alpha^-[x_\alpha^+,x_\beta^+]\big)
\end{align*}
and a similar equality for $[\tl v_j,x_\alpha^+]$ from \cite[\S3.2 and \S4.4]{GNW18}, where $c_{\g}$ is the eigenvalue of the Casimir element $C_{\g}\in \mathrm{U}(\g)$ acting on the adjoint representation. Therefore, the expression \eqref{helper4} can be rewritten as
\begin{align*}
    &\sum_{\alpha,\beta\in\Phi^+}(\alpha,\alpha_i)(\beta,\alpha_j)\big([x_\beta^-,x_\alpha^+]x_\beta^+x_\alpha^++x_\alpha^+[x_\beta^-,x_\alpha^+]x_\beta^++x_\beta^-[x_\beta^+,x_\alpha^+]x_\alpha^++x_\alpha^+x_\beta^-[x_\beta^+,x_\alpha^+]\big)
    \\-&\sum_{\alpha,\beta\in\Phi^+}(\alpha,\alpha_i)(\beta,\alpha_j)\big([x_\alpha^-,x_\beta^+]x_\alpha^+x_\beta^+ +x_\beta^+[x_\alpha^-,x_\beta^+]x_\alpha^+ + x_\alpha^-[x_\alpha^+,x_\beta^+]x_\beta^+ + x_\beta^+x_\alpha^-[x_\alpha^+,x_\beta^+]\big)\\
    +&\sum_{\alpha,\beta\in\Phi^+}(\alpha,\alpha_i)(\beta,\alpha_j)\big(x_\alpha^+[x_\alpha^+,x_\beta^+]x_\beta^++x_\alpha^+x_\beta^+[x_\alpha^+,x_\beta^+]+[x_\alpha^+,x_\beta^+]x_\beta^+x_\alpha^++x_\beta^+[x_\alpha^+,x_\beta^+]x_\alpha^+\big).
\end{align*}

We need to show that all terms cancel. 

Clearly, the terms for $\alpha=\beta$ cancel. 

Next, we show that if the terms involving exactly one negative root vector cancel. Here we treat the commutators as root vectors. There are two situations: $x_{\alpha_0}^-x_{\alpha_0+\beta_0}^+x_{\beta_0}^+$ or $x_{\beta_0}^+x_{\alpha_0}^-x_{\alpha_0+\beta_0}^+$ for any given $\alpha_0,\beta_0\in\Phi^+$ such that $\alpha_0+\beta_0\in\Phi^+$. We first consider the occurrence of $x_{\alpha_0}^-x_{\alpha_0+\beta_0}^+x_{\beta_0}^+$. Then it appears in the following 4 terms
\begin{align*}
(\beta_0,\alpha_i)(\alpha_0+\beta_0,\alpha_j)[x_{\alpha_0+\beta_0}^-,x_{\beta_0}^+]&x_{\alpha_0+\beta_0}^+x_{\beta_0}^+\\=(\beta_0,\alpha_i)(\alpha_0+\beta_0,\alpha_j)&\eta_{-\alpha_0-\beta_0,\beta_0}x_{\alpha_0}^-x_{\alpha_0+\beta_0}^+x_{\beta_0}^+,\\
-(\alpha_0+\beta_0,\alpha_i)(\beta_0,\alpha_j)[x_{\alpha_0+\beta_0}^-,x_{\beta_0}^+]&x_{\alpha_0+\beta_0}^+x_{\beta_0}^+\\
=-(\alpha_0+\beta_0,\alpha_i)(\beta_0,\alpha_j)&\eta_{-\alpha_0-\beta_0,\beta_0}x_{\alpha_0}^-x_{\alpha_0+\beta_0}^+x_{\beta_0}^+,\\
(\beta_0,\alpha_i)(\alpha_0,\alpha_j)x_{\alpha_0}^-[x_{\alpha_0}^+,x_{\beta_0}^+]x_{\beta_0}^+&=(\beta_0,\alpha_i)(\alpha_0,\alpha_j)\eta_{\alpha_0,\beta_0}x_{\alpha_0}^-x_{\alpha_0+\beta_0}^+x_{\beta_0}^+,
\\
-(\alpha_0,\alpha_i)(\beta_0,\alpha_j)x_{\alpha_0}^-[x_{\alpha_0}^+,x_{\beta_0}^+]x_{\beta_0}^+&=-(\alpha_0,\alpha_i)(\beta_0,\alpha_j)\eta_{\alpha_0,\beta_0}x_{\alpha_0}^-x_{\alpha_0+\beta_0}^+x_{\beta_0}^+.
\end{align*}
By \eqref{eta}, we have $\eta_{\alpha_0,\beta_0}+\eta_{-\alpha_0-\beta_0,\beta_0}=0$ and hence these 4 terms cancel. A similar calculation also works for the case $x_{\beta_0}^+x_{\alpha_0}^-x_{\alpha_0+\beta_0}^+$.

Finally, we consider the remaining case when all monomials are products of positive root vectors. It suffices to consider the products like $x_{\alpha_0}^+x_{\beta_0}^+x_{\alpha_0+\beta_0}^+$ for fixed $\alpha_0,\beta_0\in\Phi^+$ such that $\alpha_0+\beta_0\in\Phi^+$. Then by a similar calculation as above using \eqref{eta}, these products simplify as
\begin{align*}
&\big((\alpha_0,\alpha_i)(\beta_0,\alpha_j)-(\alpha_0,\alpha_j)(\beta_0,\alpha_i)\big)\eta_{\alpha_0,\beta_0}x_{\alpha_0}^+x_{\alpha_0+\beta_0}^+x_{\beta_0}^+\\
+\,&\big((\alpha_0,\alpha_i)(\beta_0,\alpha_j)-(\alpha_0,\alpha_j)(\beta_0,\alpha_i)\big)\eta_{\alpha_0,\beta_0}x_{\beta_0}^+x_{\alpha_0+\beta_0}^+x_{\alpha_0}^+\\
+\,&(\alpha_0,\alpha_j)(\beta_0,\alpha_i) \eta_{\alpha_0,\beta_0}\big(x_{\alpha_0}^+x_{\beta_0}^+x_{\alpha_0+\beta_0}^+ + x_{\alpha_0+\beta_0}^+x_{\beta_0}^+x_{\alpha_0}^+\big)\\
-\,&(\alpha_0,\alpha_i)(\beta_0,\alpha_j) \eta_{\alpha_0,\beta_0}\big(x_{\beta_0}^+x_{\alpha_0}^+x_{\alpha_0+\beta_0}^+ + x_{\alpha_0+\beta_0}^+x_{\alpha_0}^+x_{\beta_0}^+\big)\\
=\,&(\alpha_0,\alpha_j)(\beta_0,\alpha_i)\eta_{\alpha_0,\beta_0}\big[x_{\alpha_0}^+,[x_{\beta_0}^+,x_{\alpha_0+\beta_0}^+]\big]-(\alpha_0,\alpha_i)(\beta_0,\alpha_j)\eta_{\alpha_0,\beta_0}\big[x_{\beta_0}^+,x_{\alpha_0}^+[,x_{\alpha_0+\beta_0}^+]\big],
\end{align*}
which is clearly zero as $2(\alpha_0+\beta_0)\notin \Phi$.

Now the verification of \eqref{helper4} is complete.

\subsection{Proof of Theorem \ref{thm:embedmaintext}}\label{sec:pfthmB}
For Part (1), we have verified that $\varphi$ induces an algebra homomorphism in \S\ref{sec:pfthmB0} (we also need Appendix \ref{sec:app+} for types $\mathsf A_1$ and $\mathsf B_2\cong \mathsf C_2$). Therefore, it remains to show the homomorphism $\varphi$ is injective.

Recall the filtration on $\Y$ given by setting $\mathrm{deg}\,\xi_{i,r}=\mathrm{deg}\,x_{i,r}^\pm=r$ and the algebra isomorphism $\rho$ from \eqref{ass}. Also recall the filtration on $\Yi$ given by setting $\mathrm{deg}\,h_{i,2r+1}=2r+1$ and $\mathrm{deg}\,b_{i,r}=r$ and the algebra isomorphism $\varrho$ from \eqref{assi}. Clearly, the homomorphism $\varphi$ is a filtered algebra homomorphism and hence induces the associated graded homomorphism
\begin{align*}
\mathrm{\gr}\,\varphi:~&\mathrm{\gr}\,\Yi\to \mathrm{\gr}\,\Y,\\
&\bar{b}_{i,r}\mapsto (\bar x_{i,r}^+-(-1)^r\bar x_{i,r}^-),\quad \bar h_{i,2r+1}\mapsto 2\bar\xi_{i,2r+1}.
\end{align*}
Identifying $\mathrm{\gr}\,\Yi$ with $\mathrm{U}(\g[z]^{\check\omega})$ via \eqref{assi} and $\mathrm{\gr}\,\Y$ with $\mathrm{U}(\g[z])$ via \eqref{ass}, then $\mathrm{gr}\,\varphi$ is the natural embedding of $\mathrm{U}(\g[z]^{\check\omega})$ into $\mathrm{U}(\g[z])$. Therefore, $\mathrm{gr}\,\varphi$ is injective, implying further that $\varphi$ is injective.

Now we prove Part (2). We identify $\Yi$ as a subalgebra of $\Y$ via $\varphi$. The equation \eqref{coprohi1} follows from a straightforward calculation using \eqref{copro-efh} and \eqref{coproxii1}. Thus we have $\Delta(b_{i,0})\in \Yi\otimes \Y$ and $\Delta(h_{i,1})\in \Yi\otimes \Y$. Since $\Yi$ is generated by $h_{i,1}$ and $b_{i,0}$, we conclude that $\Delta(\Yi)\subset \Yi\otimes \Y$. Hence $\Yi$ is a right coideal subalgebra of $\Y$.

Finally, we prove Part (3). We identify $\mathrm{gr}\,\Y$ with   $\mathrm{U}(\g[z])$ via the isomorphism $\rho$ from \eqref{ass}. In the $J$ presentation, by \eqref{eq:J} the images of $x\in \g$ and $J(x)$ are given by $x$ and $xz$, respectively. Note that $\YiJ$ is generated by $b_{\alpha}=x_{\alpha}^+-x_{\alpha}^-$, and
\[
B(\xi_i)=J(\xi_i)-\tfrac14[\xi_i,C_{\mathfrak k}],\qquad B(y_\alpha)=J(x_{\alpha}^++x_{\alpha}^-)-\tfrac14[x_{\alpha}^++x_{\alpha}^-,C_{\mathfrak k}].
\]
By \eqref{ass}, the images of $b_{\alpha}$, $B(\xi_i)$ and $B(y_\alpha)$ in the associated graded of $\Y$ (considered as $\mathrm{U}(\g[z])$) are $x_{\alpha}^+-x_{\alpha}^-$, $\xi_iz$, and $(x_{\alpha}^++x_{\alpha}^-)z$, respectively. Thus the image of $\YiJ$ in the associated graded is the subalgebra $\mathrm{U}(\g[z]^{\check \omega})$. Note that it known from above that the image of $\Yi$ in the associated graded of $\Y$ is also the subalgebra $\mathrm{U}(\g[z]^{\check \omega})$. Thus, to prove that $\Yi=\YiJ$, it suffices to show that $\Yi\subset \YiJ$. Since $b_{i,0}\in \YiJ$ and $\Yi$ is generated by $b_{i,0}$ and $h_{i,1}$, this reduces to showing that $h_{i,1}\in \YiJ$. To see this, note that
\begin{align*}
&2B(\xi_i)-\tfrac{1}2\sum_{\alpha\in\Phi^+}(\alpha_i,\alpha)(b_\alpha)^2\\
=~~\,& 2J(\xi_i)- \tfrac12[\xi_i,C_{\mathfrak k}]+\tfrac{1}2\sum_{\alpha\in\Phi^+}(\alpha_i,\alpha)(x_{\alpha}^+-x_{\alpha}^-)^2\\
\overset{\eqref{Ck}}{\underset{\eqref{eq:J}}{=}}\,\,&2\xi_{i,1}-\xi_i^2+\tfrac12\sum_{\alpha\in\Phi^+}\Big((\alpha,\alpha_i)\{x_\alpha^+,x_\alpha^-\}+\tfrac12\Big[\xi_i,(x_\alpha^+-x_\alpha^-)^2\Big]+(\alpha,\alpha_i)(x_{\alpha}^+-x_{\alpha}^-)^2\Big)\\
\stackrel{\eqref{eq:relHX}}{=}\,\,&2\xi_{i,1}-\xi_i^2+\tfrac12\sum_{\alpha\in\Phi^+}(\alpha,\alpha_i)\Big(\{x_\alpha^+,x_\alpha^-\}+\tfrac12\big\{x_\alpha^++x_\alpha^-,x_\alpha^+-x_\alpha^-\big\}+(x_{\alpha}^+-x_{\alpha}^-)^2\Big)\\
=~~\,&2\xi_{i,1}-\xi_i^2+\sum_{\alpha\in\Phi^+}(\alpha,\alpha_i)(x_\alpha^+)^2=h_{i,1}.
\end{align*}
Now the proof of Theorem \ref{thm:embedmaintext} is complete.

\section{Coproduct of Drinfeld generators}\label{sec:estimate}
\subsection{Estimates of coproduct}
From now on, we regard the twisted Yangian $\Yi$ as a right coideal subalgebra of the Yangian $\Y$ via the homomorphism $\varphi$ defined by \eqref{hi1-embedding}--\eqref{bi1-embedding}. 

Our next main result is an estimate of the generators $h_{i,2r+1}$, $b_{i,r}$ in terms of the elements $\xi_{j,s}$ and $x_{j,s}^\pm$.

Set 
\begin{align*}
h_i(u)=1+\sum_{r\gge 0}h_{i,2r+1}u^{-2r-2},\qquad b_i(u)=\sum_{r\gge 0}b_{i,r}u^{-r-1}.
\end{align*}
\begin{thm}\label{thm:hb}
Let $\g$ be a simple Lie algebra which is not of type $\mathsf G_2$. We have
\begin{align}
&h_i(u)\equiv \xi_i(u)\xi_i(-u) &\pmod{\Y_{Q_+}[\![u^{-1}]\!]},\label{eq:h-est}\\
&b_i(u)\equiv \tfrac12 \{x_i^+(u),\xi_i(-u)\}+x_i^-(-u)   &\pmod{\Y_{\alpha_i+Q_+}^{\gge 0}[\![u^{-1}]\!]},\label{eq:b-est}\\
&\Delta(h_i(u))\equiv h_i(u)\otimes\xi_i(u)\xi_i(-u)  &\pmod{\Yi\otimes \Y_{Q_+}[\![u^{-1}]\!]},\label{eq:h-co-est}\\
&\Delta(b_i(u))\equiv b_i(u)\otimes\xi_i(-u)+1\otimes b_i(u) &\pmod{\Yi\otimes \Y_{Q_+}[\![u^{-1}]\!]}.\label{eq:b-co-est}
\end{align}
\end{thm}
The theorem will be proved in Section \ref{sec:proof-thmC}. For the case $\g=\mathfrak{sl}_2$, we prove a stronger result (Conjecture \ref{conj}) in Appendix \ref{sec:app} (see Proposition \ref{prop-sl2}) via the isomorphism between twisted Yangians in their R-matrix and Drinfeld presentations \cite{LWZ25GD}. Theorem~\ref{thm:hb} is the twisted Yangian counterpart of \cite[Corollary~9.16]{Prz23} for affine $\imath$quantum groups of split type $\mathsf A$, and of \cite[Theorem~8.1]{LP25} for affine $\imath$quantum groups of split types $\mathsf{BCD}$.

A similar estimate of the form \eqref{eq:h-est} was conjectured in \cite{WZ23} for affine $\imath$quantum groups of split type $\sf A$ in a stronger version. 

\begin{conj}[{\cite{WZ23}}]\label{conj}
We have
\begin{align*}
h_i(u)&\equiv \xi_i(u)\xi_i(-u)  &\pmod{\Y_{Q_+}^{\gge 0}[\![u^{-1}]\!]},\\
\Delta(h_i(u))&\equiv h_i(u)\otimes\xi_i(u)\xi_i(-u) &\pmod{\Yi\otimes \Y^{\gge 0}_{Q_+}[\![u^{-1}]\!]}.
\end{align*}
\end{conj}

Proposition \ref{prop-sl2} gives a positive answer to Conjecture \ref{conj} for the case $\g=\mathfrak{sl}_2$. Note that our approach can be generalized to the affine $\imath$quantum group of split type $\mathsf A_1$ in a straightforward way using the isomorphism between affine $\imath$quantum groups of split type $\mathsf A$ in the R-matrix and Drinfeld presentations \cite{Lu24}. Thus it gives a simpler proof of \cite[Theorem~7.5]{Prz23}.

\subsection{Restriction modules}
Let $\mathscr P_\I:=(1+u^{-2}\bC[\![u^{-2}]\!])^{\I}$ denote the subset of $\I$-tuple of power series in $u^{-1}$ with constant term 1. We call an element in $\mathscr P_{\I}$ an $^\imath\bm\ell$-weight. We write an $^\imath\bm\ell$-weight in the form $\bla=(\la_i(u))_{i\in \I}$, 
where
\[
\la_i(u)=1+\sum_{r>0}\la_{i,2r-1}u^{-2r}.
\]
For a finite-dimensional $\Yi$-module $\mc M$ and $\bla\in\mathscr P_{\I}$, the \textit{$^\imath\bm\ell$-weight space} of $^\imath\bm\ell$-weight $\bla$ is a subspace of $\mc M$ defined by
\[
\mc M_{\bla}:=\big\{v\in \mc M\mid \forall~ i\in\I \text{ and }r>0, \exists~ p>0 \text{ such that } (h_{i,2r-1}-\la_{i,2r-1})^pv=0\big\}.
\]
If $\mc M_{\bla}\ne 0$, then we say $\bla$ is an $^\imath\bm\ell$-weight of $\mc M$.

The estimate \eqref{eq:h-est} can be used to obtain the spectrum of $h_i(u)$ acting on a finite-dimensional $\Y$-module regarded as a $\Yi$-module via restriction. 

For a monic polynomial $P(u)$ in $u$, denote
\[
P^-(u)=(-1)^{\deg P(u)}P(-u).
\]
Note that $P^-(u)$ is the monic polynomial whose roots are opposite to that of $P(u)$. 

The following is an immediate corollary of \cite[Theorem~1]{Kn95} and \eqref{eq:h-est} as $h_i(u)$ and $\xi_i(u)\xi_i(-u)$ share the same eigenvalues.

\begin{cor}\label{cor:cha}
Let $\g$ be a simple Lie algebra which is not of type $\mathsf G_2$. Let $\mathscr V$ be a finite-dimensional $\Y$-module regarded as a $\Yi$-module via restriction. Then an $^\imath\bm\ell$-weight $\bm\la$ of $\mathscr V$ has the form
\[
\la_i(u)=\frac{\Xi_i(u+\hf d_i)\Xi_i^-(u-\hf d_i)}{\Xi_i(u-\hf d_i)\Xi_i^-(u+\hf d_i)},\qquad \text{ for }i\in \I,
\]
where $\Xi_i(u)$ is a monic polynomial in $u$.
\end{cor}

\subsection{Proof of Theorem \ref{thm:hb}}\label{sec:proof-thmC}
We record some equalities that will be used in the proof:
\begin{align}
&[x_i^+,x_i^-(u)]=\xi_i(u)-1,\label{xi+xi-}\\
&[x_i^-,x_i^-(u)]=\tfrac12(\alpha_i,\alpha_i)(x_i^-(u))^2,\label{xixi-}\\
&[\tl\xi_{i,1},x_{j}^\pm(u)]=\pm(\alpha_i,\alpha_j)(ux_j^\pm(u)-x_{j}^\pm),\label{xixj1}\\
&[\xi_i(u),x_{i,0}^-]=\tfrac12(\alpha_i,\alpha_i)\{\xi_i(u),x_i^-(u)\},\label{xiuxi0}\\
&[h_{i,1},b_j(u)]=2(\alpha_i,\alpha_j)(ub_j(u)-b_{j,0}),\label{hbcom1}\\
&[b_{i,0},b_i(u)]=h_i(u)-1-\tfrac12(\alpha_i,\alpha_i)(b_i(u))^2.\label{bi0biu}
\end{align}
The equalities \eqref{xi+xi-}, \eqref{xixj1}, and \eqref{hbcom1} follow from \eqref{eq:relXX},  \eqref{tlxicom}, and \eqref{eq:hi1bjr-new}, respectively. The equalities \eqref{xixi-} and \eqref{xiuxi0} can be found in \cite[\S2.4]{GTL16} while the equality \eqref{bi0biu} is obtained from \cite[Lemma~4.4]{LWZ25affine}.

\subsubsection{} We first establish \eqref{eq:b-est} whose component-wise formula is given by
\beq
b_{i,r}\equiv x_{i,r}^+-(-1)^rx_{i,r}^--\tfrac12\sum_{0\lle s<r}(-1)^s \{x_{i,r-s-1}^+,\xi_{i,s} \}.
\eeq
Here $\equiv$ stands for equality modulo $\Y_{\alpha_i+Q_+}^{\gge 0}[\![u^{-1}]\!]$. We prove it by induction on $r$. The base cases $r=0,1$ are clear from \eqref{bi0-embedding} and \eqref{bi1-embedding}. 

Note that by \eqref{hi1-embedding} we have $[h_{i,1},\Y_{\alpha_i+Q_+}^{\gge 0}]\subset \Y_{\alpha_i+Q_+}^{\gge 0}$. Together with \eqref{hbcom1}, it is not hard to see that the induction step is reduced to showing
\beq\label{helper5}
\begin{split}
\big[h_{i,1},&\,\tfrac12 \{x_i^+(u),\xi_i(-u)\}+x_i^-(-u)\big]\\
&\equiv 2(\alpha_i,\alpha_i)u\big(\tfrac12\{x_i^+(u),\xi_i(-u)\}+x_i^-(-u)\big)-2(\alpha_i,\alpha_i)(x_i^+-x_i^-).
\end{split}
\eeq
By \eqref{tlxidef} and \eqref{hi1-embedding}, the LHS of \eqref{helper5} is equal to
\begin{align*}
&\Big[2\tl\xi_{i,1}+\sum_{\alpha\in\Phi^+}(\alpha,\alpha_i)(x_\alpha^+)^2,\tfrac12 \{x_i^+(u),\xi_i(-u)\}+x_i^-(-u)\Big]\\
\stackrel{\eqref{xixj1}}{\equiv}\,&(\alpha_i,\alpha_i)\Big(\big\{ux_i^{+}(u)-x_i^+,\xi_i(-u)\}+ 2\big(ux_i^-(-u)+x_i^-\big)+\big\{[x_i^+,x_i^-(-u)],x_i^+\big\}\Big)\\
\stackrel{\eqref{xi+xi-}}{=}\,&(\alpha_i,\alpha_i)\Big(\big\{ux_i^{+}(u)-x_i^+,\xi_i(-u)\}+ 2\big(ux_i^-(-u)+x_i^-\big)+\big\{\xi_i(-u)-1,x_i^+\big\}\Big)\\
=\hskip0.22cm&2(\alpha_i,\alpha_i)u\big(\tfrac12\{x_i^+(u),\xi_i(-u)\}+x_i^-(-u)\big)-2(\alpha_i,\alpha_i)(x_i^+-x_i^-)
\end{align*}
as required.

\subsubsection{} 
For \eqref{eq:b-co-est}, the corresponding component-wise formula is given by 
\[
\Delta(b_{i,r})\equiv \sum_{0\lle s<r}(-1)^{s+1} b_{i,r-s-1}\otimes \xi_{i,s}+b_{i,r}\otimes 1+1\otimes b_{i,r}.
\]
Here $\equiv$ stands for equality modulo $\Yi\otimes \Y_{Q_+}[\![u^{-1}]\!]$. We prove it again by induction on $r$. The base case $r=0$ is clear while the case $r=1$ follows from \eqref{helper98}. 

Note that by \eqref{coprohi1} we have $[\Delta(h_{i,1}),\Yi\otimes \Y_{Q_+}]\subset \Yi\otimes \Y_{Q_+}$. Together with the relation 
\begin{align*}
\big[\Delta(h_{i,1}),\Delta(b_i(u))\big]=2(\alpha_i,\alpha_i)\big(u\Delta(b_i(u))-\Delta(b_{i,0})\big)
\end{align*}
that follows from \eqref{hbcom1}, 
it is not hard to see that the induction step is reduced to proving
\beq\label{helper21}
\begin{split}
\Big[h_{i,1}\otimes 1&+ 1\otimes h_{i,1}+2\sum_{\alpha\in\Phi^+}(\alpha,\alpha_i)\,b_\alpha\otimes x_\alpha^+,b_i(u)\otimes \xi_i(-u)+1\otimes b_i(u)\Big]\\
\equiv &~2(\alpha_i,\alpha_i)\Big(u\big(b_i(u)\otimes \xi_i(-u)+1\otimes b_i(u)\big)-(b_{i,0}\otimes 1+1\otimes b_{i,0})\Big).
\end{split}
\eeq
We list all the terms from the LHS of \eqref{helper21} as follows:
\begin{align*}
&[h_{i,1}\otimes 1, b_i(u)\otimes \xi_i(-u)]\stackrel{\eqref{hbcom1}}{=}2(\alpha_i,\alpha_i)(ub_i(u)-b_{i,0})\otimes \xi_i(-u),\\
&[h_{i,1}\otimes 1, 1\otimes b_i(u)]=0,\qquad [1\otimes h_{i,1}, b_i(u)\otimes \xi_i(-u)]\stackrel{\eqref{hi1-embedding}}{\equiv}0,\\
&[1\otimes h_{i,1}, 1\otimes b_i(u)]\stackrel{\eqref{hbcom1}}{=}2(\alpha_i,\alpha_i) 1\otimes (ub_i(u)-b_{i,0}),\\
&[b_\alpha\otimes x_\alpha^+,b_i(u)\otimes \xi_i(-u)]\equiv 0,\\
&[b_\alpha\otimes x_\alpha^+,1\otimes b_i(u)]\stackrel{\eqref{eq:b-est}}{\equiv} 0,\qquad \text{ if }\alpha\ne \alpha_i,\\
&[b_{i,0}\otimes x_i^+,1\otimes b_i(u)]\overset{\eqref{eq:b-est}}{\underset{\eqref{xi+xi-}}{\equiv}}b_{i,0}\otimes (\xi_i(-u)-1).
\end{align*}
Summing them up (with suitable multiples), we easily deduce \eqref{helper21}.

\subsubsection{} Next we prove \eqref{eq:h-est}. Recall that $d_i=\tfrac12(\alpha_i,\alpha_i)$ and let $\equiv$ stand for equality modulo $\Y_{Q_+}[\![u^{-1}]\!]$. By \eqref{bi0biu}, we have
\begin{align*}
h_i(u) =\,\,\, & [b_{i,0},b_i(u)]+d_i(b_i(u))^2+1\\
\stackrel{\eqref{eq:b-est}}{\equiv}\, & \big[x_i^+-x_i^-,\tfrac12 \{x_i^+(u),\xi_i(-u)\}+x_i^-(-u)\big]+d_i\big(\tfrac12 \{x_i^+(u),\xi_i(-u)\}+x_i^-(-u)\big)^2+1\\
\equiv\,\,\, & [x_i^+,x_i^-(-u)]+\tfrac12\{[x_i^+(u),x_i^-],\xi_i(-u)\}+\tfrac12\{x_i^+(u),[\xi_i(-u),x_i^-]\}-[x_i^-,x_i^-(-u)]\\
&+d_i\big(\tfrac12\{x_i^+(u)\xi_i(-u),x_i^-(-u)\}+\tfrac12\{\xi_i(-u)x_i^+(u),x_i^-(-u)\}+(x_i^-(-u))^2\big)+1\\
\overset{\eqref{xi+xi-}}{\underset{\eqref{xixi-}}{=}}&\,\xi_i(-u)+\tfrac12\{\xi_i(u)-1,\xi_i(-u)\}+\tfrac12\{x_i^+(u),[\xi_i(-u),x_i^-]\}\\
&+d_i\big(\tfrac12\{x_i^+(u)\xi_i(-u),x_i^-(-u)\}+\tfrac12\{\xi_i(-u)x_i^+(u),x_i^-(-u)\}\big)\\
\stackrel{\eqref{xiuxi0}}{=}&\,\xi_i(u)\xi_i(-u)-\tfrac12d_i\{x_i^+(u),\xi_i(-u)x_i^-(-u)+x_i^-(-u)\xi_i(-u)\}\\
&+d_i\big(\tfrac12\{x_i^+(u)\xi_i(-u),x_i^-(-u)\}+\tfrac12\{\xi_i(-u)x_i^+(u),x_i^-(-u)\}\big)\\
=\hskip 0.18cm&\, \xi_i(u)\xi_i(-u)-\tfrac12d_i\big[[x_i^+(u),x_i^-(-u)],\xi_i(-u)\big]\overset{\eqref{eq:relHH}}{\underset{\eqref{eq:relXX}}{=}}\xi_i(u)\xi_i(-u),
\end{align*}
completing the proof of \eqref{eq:h-est}. 

\subsubsection{} 
Finally, we prove \eqref{eq:h-co-est}. We start with the following lemma. For $i\in \I$, set
\beq\label{Qi+}
Q_{i,+}:=\Big\{\sum_{j\in\I}k_j\alpha_j~\Big|~k_i>0,~k_j\in\bZ \text{ for } j\in \I \Big\},
\eeq
\be 
Q_{i,-}:=\Big\{\sum_{j\in\I}k_j\alpha_j~\Big|~k_i\lle 0,~k_j\in\bZ \text{ for } j\in \I \Big\}.
\ee
Then
\beq\label{dec-Q}
\Y=\Y_{Q_{i,+}}\oplus \Y_{Q_{i,-}}.
\eeq

\begin{lem}\label{lem:cap=0}
We have $\Yi\cap \Y_{Q_{i,+}}=\{0\}$ for all $i\in \I$.
\end{lem}
\begin{proof}
Recall from \S\ref{sec:pfthmB} that the filtration on $\Yi$ coincides with the filtration induced from that on $\Y$ if we regard $\Yi$ as a subalgebra of $\Y$. Also note that the algebra $\mathrm{U}(\g[z])$ is $Q$-graded in the obvious way and the associated graded map preserves the $Q$-grading, see \eqref{ass}. By passing to the associated graded, it suffices to show that $\mathrm{U}(\g[z]^{\check{\omega}})\cap \mathrm{U}(\g[z])_{Q_{i,+}}=\{0\}$, where $\mathrm{U}(\g[z])_{Q_{i,+}}$ is the associated graded image of $\Y_{Q_{i,+}}$. Note that $\mathrm{U}(\g[z])_{Q_{i,+}}$ can also be described explicitly as the definition of $\Y_{Q_{i,+}}$ from \S\ref{sec:Yang} by setting $\deg \xi_iz^r=0$ and $\deg x_{i}^\pm z^r=\pm \alpha_i$. We shall show that if $\mathcal E\in \mathrm{U}(\g[z]^{\check{\omega}})$ is nonzero, then $\mathcal E$ does not belong to $\mathrm{U}(\g[z])_{Q_{i,+}}$. 

Fix an arbitrary total order $\prec$ on $\I\times\bN$ and also an arbitrary total order $\prec$ on the set of positive roots $\Phi^+$. Lexicographically extend the total order $\prec$ on $\Phi^+$ and the natural order on $\bN$ to the set $\Phi^+\times\bN$, which we denote by $\prec$ again. 

Let $\bar h_{i,2r+1}:=2\xi_{i}z^{2r+1}$ and $\bar b_{\alpha,r}:=(x_\alpha^+-(-1)^rx_\alpha^-)z^r$ for $i\in \I$, $r\in \bN$, and $\alpha\in \Phi^+$. By the PBW theorem, the nonzero element $\mathcal E\in \mathrm{U}(\g[z]^{\check{\omega}})$ can be written uniquely as a finite sum of the following form:
\beq\label{helper97}
\sum_{\bm m,\bm n}k_{\bm m,\bm n}\prod_{(i,r)\in \I\times\bN}\bar h_{i,2r+1}^{m_{i,r}}\prod_{(\alpha,s)\in \Phi^+\times\bN}\bar b_{\alpha,s}^{n_{\alpha,s}},
\eeq
summed over $\bm m =(m_{i,r})_{(i,r)\in \I\times\bN}$ and $\bm n=(n_{\alpha,s})_{(\alpha,s)\in \Phi^+\times\bN}$ that are tuples of natural numbers with only finitely many nonzero numbers. Here $k_{\bm m,\bm n}$ are complex numbers and the orders of the products in \eqref{helper97} are taken with respect to the total order $\prec$ we choose on the index sets.

Extend the total order $\prec$ on $\Phi^+\times \bN$ to the set of root vectors $\bar x_{\alpha,r}^\pm:=x_\alpha^\pm z^r $ of $\g[z]$ (or equivalently the set $\Phi\times \bN$) by the rule: 
\[
\text{for } \alpha,\beta\in \Phi^+ \text{ and } r,s\in \bN,\text{ if }(\alpha,r)\prec(\beta,s), \text{ then }\bar x_{\alpha,r}^-\prec\bar x_{\alpha,r}^+\prec\bar x_{\beta,s}^-\prec\bar x_{\beta,s}^+.
\]

Expand all $\bar b_{\alpha,s}^{n_{\alpha,s}}$ from a monomial in \eqref{helper97} in terms of  $\bar x_{\alpha,r}^\pm$ and rewrite the result in terms of ordered monomials in $\bar\xi_{i,r}:=\xi_iz^r$ and $\bar x_{\alpha,r}^\pm$, where we always put $\bar\xi_{i,r}$ to the left of $\bar x_{\alpha,r}^\pm$. Then such a monomial in \eqref{helper97} can be written as
\begin{align*}
&k_{\bm m,\bm n}\prod_{(i,r)\in \I\times\bN}\bar h_{i,2r+1}^{m_{i,r}}\prod_{(\alpha,s)\in \Phi^+\times\bN}(\bar x_{\alpha,s}^+)^{n_{\alpha,s}}\\+\,&(-1)^*k_{\bm m,\bm n}\prod_{(i,r)\in \I\times\bN}\bar h_{i,2r+1}^{m_{i,r}}\prod_{(\alpha,s)\in \Phi^+\times\bN}(\bar x_{\alpha,s}^-)^{n_{\alpha,s}}+\cdots,
\end{align*}
where $*$ is a certain integer and $\cdots$ stands for a certain sum of ordered monomials whose weights are strictly less than $\alpha$ and greater than $-\alpha$ (by that we mean the order $>$ defined on $Q$ by the rule: $\alpha>\beta$ if $\alpha-\beta\in Q_+$).

Pick a pair $(\bm m',\bm n')$ such that $\sum n'_{\alpha,s}\alpha$ summed over $(\alpha,s)\in \Phi^+\times\bN$ is maximal with respect to the order $>$ among all $(\bm m,\bm n)$ such that  $k_{\bm m,\bm n}\ne 0$ in \eqref{helper97}. Then the ordered monomial 
\[
(-1)^*k_{\bm m',\bm n'}\prod_{(i,r)\in \I\times\bN}\bar h_{i,2r+1}^{m'_{i,r}}\prod_{(\alpha,s)\in \Phi^+\times\bN}(\bar x_{\alpha,s}^-)^{n'_{\alpha,s}}
\]
cannot be canceled by other ordered monomials after rewriting \eqref{helper97}. Therefore, $\mathcal E$ does not belong to $\mathrm{U}(\g[z])_{Q_{i,+}}$, completing the proof.
\end{proof}

Now we are ready to prove \eqref{eq:h-co-est}. It follows from \eqref{eq:h-est} that $h_i(u)=\xi_i(u)\xi_i(-u)+\theta(u)$, where $\theta(u)\in \Y_{Q_+}[\![u^{-1}]\!]$. It follows from Lemma \ref{lem:copro} that $\Delta(\xi_i(u))=\xi_i(u)\otimes \xi_i(u)+Z(u)$ where $Z(u)\in \Y_{Q_-}\otimes\Y_{Q_+}[\![u^{-1}]\!]$. Therefore, we have
\begin{align*}
\Delta(h_i(u))&=\big(\xi_i(u)\otimes \xi_i(u)+Z(u)\big)\big(\xi_i(-u)\otimes \xi_i(-u)+Z(-u)\big)+\Delta(\theta(u))\\
&=\xi_i(u)\xi_i(-u)\otimes \xi_i(u)\xi_i(-u)+\big(\xi_i(u)\otimes \xi_i(u)\big)Z(-u)\\&~~~\,+Z(u)\big(\xi_i(-u)\otimes \xi_i(-u)\big)+Z(u)Z(-u)+\Delta(\theta(u))\\
&=h_i(u)\otimes \xi_i(u)\xi_i(-u)-\theta(u)\otimes \xi_i(u)\xi_i(-u)+\big(\xi_i(u)\otimes \xi_i(u)\big)Z(-u)\\&~~~\,+Z(u)\big(\xi_i(-u)\otimes \xi_i(-u)\big)+Z(u)Z(-u)+\Delta(\theta(u)).
\end{align*}
Denote $\Theta(u)=\Delta(h_i(u))-h_i(u)\otimes \xi_i(u)\xi_i(-u)$. Then 
\begin{align*}
\Theta(u)&=\Delta(\theta(u))-\theta(u)\otimes \xi_i(u)\xi_i(-u)\\&~~~\,+\big(\xi_i(u)\otimes \xi_i(u)\big)Z(-u)+Z(u)\big(\xi_i(-u)\otimes \xi_i(-u)\big)+Z(u)Z(-u).
\end{align*}
Since $\Yi$ is a right coideal subalgebra of $\Y$, we have $\Theta(u)\in \Yi\otimes \Y[\![u^{-1}]\!]$.
Given any $j\in\I$, by \eqref{dec-Q} we can write 
\[
\Theta(u)=\Theta_{j,+}(u)+\Theta_{j,-}(u),\qquad \Theta_{j,\pm}(u)\in \Yi\otimes\Y_{Q_{j,\pm}}[\![u^{-1}]\!].
\]
Clearly, only $\Delta(\theta(u))-\theta(u)\otimes \xi_i(u)\xi_i(-u)$ contributes to $\Theta_{j,-}(u)$. Since $\theta(u)\in \Y_{Q_+}[\![u^{-1}]\!]$, we conclude that $\Theta_{j,-}(u)\in \Y_{Q_{j,+}}\otimes\Y_{Q_{j,-}}[\![u^{-1}]\!]$. Thus $\Theta_{j,-}(u)\in \big(\Yi\cap\Y_{Q_{j,+}}\big)\otimes\Y_{Q_{j,-}}[\![u^{-1}]\!]$. It follows from Lemma \ref{lem:cap=0} that $\Theta_{j,-}(u)=0$. Therefore $\Theta(u)=\Theta_{j,+}(u)\in \Yi\otimes\Y_{Q_{j,+}}[\![u^{-1}]\!]$. Since $j\in \I$ is arbitrary, we further have  $\Theta(u)\in \Yi\otimes\Y_{Q_{+}}[\![u^{-1}]\!]$, completing the proof of \eqref{eq:h-co-est}.

\appendix
\section{R-matrix presentation}\label{sec:app+}
\subsection{Type $\mathsf A_1$}\label{sec:app}
We prove Part (1) of Theorem \ref{thm:embedmaintext} (Proposition \ref{prop-sl2-2}) and Theorem \ref{thm:hb} (Proposition \ref{prop-sl2}) for the rank 1 case $\g=\mathfrak{sl}_2$. Our approach also works for the affine $\imath$quantum group case with the help of the isomorphism between affine $\imath$quantum group of split type $\mathsf A$ in R-matrix and Drinfeld presentations established in \cite{Lu24}. This approach would greatly simplify the proof of \cite[Theorem~7.5]{Prz23} and gives a stronger form.

We first recall Yangians and twisted Yangians in R-matrix presentation (we shall only use relations in generating series form) from \cite{MNO96,Mo07}.

The \textit{Yangian} $\Y_{\mathscr R}$ corresponding to the Lie algebra $\gl_2$ is a unital associative algebra with generators $t_{ij}^{(r)}$, where $1\lle i,j\lle 2$ and $r\in\bZ_{>0}$, and the defining relations 
\beq\label{RTT}
(u-v)[t_{ij}(u),t_{kl}(v)]=t_{kj}(u)t_{il}(v)-t_{kj}(v)t_{il}(u).
\eeq
Here we have used the generating series in an indeterminate $u$
\[
t_{ij}(u)=  \delta_{ij}+t_{ij}^{(1)}u^{-1}+t_{ij}^{(2)}u^{-2}+\cdots.
\]
Introduce new generating series via Gauss decomposition, 
\begin{align*}
t_{11}(u)&=D_1(u), &t_{22}(u)=D_2(u)+F(u)D_1(u)E(u),\\
t_{12}(u)&=D_1(u)E(u),& t_{21}(u)=F(u)D_1(u).
\end{align*}

It has been established in \cite{BK05} that the Yangian $\Y(\mathfrak{sl}_2)$ in Drinfeld presentation can be identified as a subalgebra of $\Y_{\mathscr R}$ via the following correspondence:
\beq\label{eq:Y2-R}
x_1^+(u)=F(u-\tfrac{1}{2}),\quad x_1^-(u)=E(u-\tfrac{1}{2}),\quad
\xi_1(u)= D_{1}(u-\tfrac{1}{2})^{-1}D_{2}(u-\tfrac{1}{2}). 
\eeq

The \textit{quantum} determinant $\mathrm{qdet}\,T(u)$ of the Yangian $\Y_{\mathscr R}$ is defined by the first equality which also satisfies the second equality,
\beq\label{eq:qdet}
\mathrm{qdet}\,T(u)=t_{11}(u)t_{22}(u-1)-t_{21}(u)t_{12}(u-1)=D_1(u)D_2(u-1).
\eeq
Moreover, the coefficients of $\mathrm{qdet}\,T(u)$ as a series in $u^{-1}$ are free generators of the center of $\Y_{\mathscr R}$.

The coproduct of $\Y_{\mathscr R}$ is given by 
\[
\Delta(t_{ij}(u))=t_{i1}(u)\otimes t_{1j}(u)+t_{i2}(u)\otimes t_{2j}(u),\qquad i,j=1,2.
\]
This is compatible with the coproduct \eqref{copro-efh}--\eqref{copro-xi-} in Drinfeld presentation under the identification \eqref{eq:Y2-R}. In addition, the quantum determinant is a group-like element,
\beq\label{eq:copro-R1}
\Delta(\mathrm{qdet}\,T(u))=\mathrm{qdet}\,T(u)\otimes \mathrm{qdet}\,T(u).
\eeq

For $1\lle i,j\lle 2$, we introduce the generating series 
\beq\label{siju}
s_{ij}(u)=  \delta_{ij}+s_{ij}^{(1)}u^{-1}+s_{ij}^{(2)}u^{-2}+\cdots.
\eeq
The \textit{twisted Yangian} $\Yi_{\mathscr R}$ corresponding to the Lie algebra $\mathfrak o_2$ is the unital associative algebra  with generators $s_{ij}^{(r)}$, where $1\lle i,j\lle 2$ and $r\in\bZ_{>0}$, whose generating series \eqref{siju} satisfy the following quaternary relation:
\beq\label{quater u}
\begin{split}
(u^2-v^2)[s_{ij}(u),s_{kl}(v)]=&\, (u+v)(s_{kj}(u)s_{il}(v)-s_{kj}(v)s_{il}(u))\\
- &\, (u-v)(s_{ik}(u)s_{jl}(v)-s_{ki}(v)s_{lj}(u))\\
&\, \qquad \ \  +  s_{ki}(u)s_{jl}(v)-s_{ki}(v)s_{jl}(u).
\end{split}
\eeq
and the symmetry relations
\beq\label{sym u}
s_{ji}(-u)=  s_{ij}(u)+\frac{s_{ij}(u)-s_{ij}(-u)}{2u}.
\eeq

Introduce new generating series via Gauss decomposition, 
\begin{align*}
s_{11}(u)&=d_1(u), &s_{22}(u)=d_2(u)+f(u)d_1(u)e(u),\\
s_{12}(u)&=d_1(u)e(u),&
s_{21}(u)=f(u)d_1(u).
\end{align*}
It is shown in \cite[Lemma~4.1]{LWZ25GD} that $f(u-\tfrac{1}{2})=e(-u-\tfrac{1}{2})$. It has been established in \cite{LWZ25GD} that the twisted Yangian $\Yi(\mathfrak{sl}_2)$ in Drinfeld presentation can be identified as a subalgebra of $\Yi_{\mathscr R}$ via the following correspondence:
\beq\label{eq:Y2i-R}
b_1(u)=f(u-\tfrac{1}{2}),\qquad 
h_1(u)= d_{1}(u-\tfrac{1}{2})^{-1}d_{2}(u-\tfrac{1}{2}). 
\eeq

The twisted Yangian $\Yi_{\mathscr R}$ can be identified as a subalgebra of $\Y_{\mathscr R}$ via
\beq\label{eq:embed-R}
s_{ij}(u)\mapsto t_{1i}(-u)t_{1j}(u)+t_{2i}(-u)t_{2j}(u).
\eeq
This is the version used in \cite[\S2.2]{LPTTW25}, which is different from the one in \cite{Mo07,LWZ25GD}, in order to make it a \textit{right} coideal subalgebra. Specifically, as a subalgebra of $\Y_{\mathscr R}$, we have
\beq\label{s-copro}
\Delta(s_{ij}(u))=\sum_{a,b=1}^2 s_{ab}(u)\otimes t_{ai}(-u)t_{bj}(u).
\eeq

The \textit{Sklyanin} determinant $\mathrm{sdet}\,S(u)$ of the twisted Yangian $\Yi_{\mathscr R}$ is defined by the first equality which also satisfies the second equality,
\beq\label{eq:sdet}
\mathrm{sdet}\,S(u)=s_{11}(-u)s_{22}(u-1)-s_{12}(-u)s_{12}(u-1)=d_1(u)d_2(u-1).
\eeq
Moreover, $\mathrm{sdet}\,S(u+\hf)$ is an even series in $u^{-1}$ and its coefficients are free generators of the center of $\Yi_{\mathscr R}$. Considering $\Yi_{\mathscr R}$ as a subalgebra of $\Y_{\mathscr R}$ via \eqref{eq:embed-R}, the Sklyanin determinant satisfies
\beq\label{eq:s=qdet}
\mathrm{sdet}\,S(u)=\mathrm{qdet}\,T(u)\mathrm{qdet}\,T(-u+1)
\eeq
and hence by \eqref{eq:copro-R1} is a group-like element as well,
\beq\label{eq:copro-R2}
\Delta(\mathrm{sdet}\,S(u))=\mathrm{sdet}\,S(u)\otimes \mathrm{sdet}\,S(u).
\eeq

\begin{prop}\label{prop-sl2}
Conjecture \ref{conj} holds for the case $\g=\mathfrak{sl}_2$.
\end{prop}
\begin{proof}
Let $\Y_{\mathscr R}^{\gge 0}$ be the subalgebra generated by the coefficients of $D_1(u),D_2(u),F(u)$ and $\mathscr I_{\mathscr R}^{\gge 0}$ be the two-sided ideal of $\Y_{\mathscr R}^{\gge 0}$ generated by coefficients of $F(u)$. The notion $\equiv$ below stands for equality modulo the ideal $\mathscr I_{\mathscr R}^{\gge 0}[\![u^{-1}]\!]$ in $\Y_{\mathscr R}^{\gge 0}[\![u^{-1}]\!]$.

Note that by \eqref{eq:embed-R}, we have
\begin{align*}
d_1(u)&=s_{11}(u)=t_{11}(-u)t_{11}(u)+t_{21}(-u)t_{21}(u)\\
&=D_1(-u)D_1(u)+F(-u)D_1(-u)F(u)D_1(u)\equiv D_1(-u)D_1(u).
\end{align*}
It follows that
\beq\label{d1equiv}
d_1(u)^{-1} \equiv D_1(-u)^{-1} D_1(u)^{-1}.
\eeq
By \eqref{eq:qdet}, \eqref{eq:sdet}, and \eqref{eq:s=qdet}, we have
\[
d_1(u)d_2(u-1)=D_1(u)D_2(u-1)D_1(-u+1)D_2(-u).
\]
Combining it with \eqref{d1equiv}, we find that
\beq\label{d2equiv}
d_2(u)\equiv D_1(-u-1)^{-1}D_1(-u)D_2(-u-1)D_2(u).
\eeq
Thus by \eqref{d1equiv}--\eqref{d2equiv}, we have
\begin{align*}
h_1(u)& \stackrel{\eqref{eq:Y2i-R}}{=}d_1(u-\hf)^{-1}d_2(u-\hf)\\
&\overset{\eqref{d1equiv}}{\underset{\eqref{d2equiv}}{\equiv}} D_1(-u-\hf)^{-1}D_1(u-\hf)^{-1}D_2(-u-\hf)D_2(u-\hf)\stackrel{\eqref{eq:Y2-R}}{=}\xi_1(u)\xi_1(-u),
\end{align*}
completing the proof of the first equality of Conjecture \ref{conj}.

Then we proceed to verify the coproduct formula. The operation $\equiv$ below stands for equality modulo the ideal $\Yi_{\mathscr R}\otimes \mathscr I_{\mathscr R}^{\gge 0}[\![u^{-1}]\!]$ in $\Yi_{\mathscr R}\otimes \Y_{\mathscr R}^{\gge 0}[\![u^{-1}]\!]$. By \eqref{s-copro}, we have
\begin{align*}
\Delta(d_1(u))&=\Delta(s_{11}(u))=\sum_{a,b=1}^2 s_{ab}(u)\otimes t_{a1}(-u)t_{b1}(u)\\
&\equiv s_{11}(u)\otimes t_{11}(-u)t_{11}(u) = d_1(u)\otimes D_1(-u)D_1(u).
\end{align*}
It follows that
\beq\label{d1equivco}
\Delta(d_1(u)^{-1}) \equiv d_1(u)^{-1}\otimes D_1(-u)^{-1} D_1(u)^{-1}.
\eeq
By \eqref{eq:qdet}--\eqref{eq:copro-R1} and \eqref{eq:sdet}--\eqref{eq:copro-R2}, we have
\begin{align*}
\Delta(d_1(u)d_2(u-1))=d_1(u)d_2(u-1)\otimes D_1(u)D_2(u-1)D_1(-u+1)D_2(-u).
\end{align*}
Combining it with \eqref{d1equivco}, we obtain
\beq
\Delta(d_2(u))\equiv d_2(u)\otimes D_1(-u-1)^{-1}D_1(-u)D_2(-u-1)D_2(u).
\eeq
Then a similar calculation as in the first part shows that $\Delta(h_1(u))\equiv h_1(u)\otimes \xi_1(u)\xi_1(-u)$, completing the proof of the second equality of Conjecture \ref{conj}.
\end{proof}

Finally, we establish Part (1) of Theorem \ref{thm:embedmaintext} for type $\mathsf A_1$. For that purpose, we need the coefficients of the generating series. For any series $\mathsf X(u)$ in $u^{-1}$ from $\Y_{\mathscr R}$ and $\Yi_{\mathscr R}$ in R-matrix presentation and their Gauss decomposition, denote by $\mathsf X^{(r)}$ the coefficients of $u^{-r}$ for $r\in \bN$.

Since it is already known in \cite{LWZ25GD} that $\Yi\subset  \Yi_{\mathscr R}$ is an coideal subalgebra of $\Y$, it suffices to match the generators as in \eqref{hi1-embedding}--\eqref{bi0-embedding}.

\begin{prop}\label{prop-sl2-2}
Under the identification described above, we have 
\[
b_{1,0}=x_{1,0}^+-x_{1,0}^-,\qquad h_{1,1}=2\xi_{1,1}-\xi_{1,0}^2+2(x_{1,0}^+)^2.
\]
\end{prop}
\begin{proof}
We have
\beq\label{b10app}
b_{1,0}\stackrel{\eqref{eq:Y2i-R}}{=} f^{(1)}=s_{21}^{(1)}\stackrel{\eqref{eq:embed-R}}{=}t_{21}^{(1)}-t_{12}^{(1)}=F^{(1)}-E^{(1)}\stackrel{\eqref{eq:Y2-R}}{=}x_{1,0}^+-x_{1,0}^-,
\eeq
proving the first equality.

It follows from \eqref{eq:Y2-R} that
\beq\label{app1}
\begin{split}
&\xi_{1,0}=D_2^{(1)}-D_1^{(1)},\\
&\xi_{1,1}=D_2^{(2)}-D_1^{(2)}+(D_1^{(1)})^2+\hf D_2^{(1)}-\hf D_1^{(1)}-D_1^{(1)}D_2^{(1)}.
\end{split}
\eeq
Using \eqref{eq:embed-R} and expressing $t_{ij}(u)$ and $s_{ij}(u)$ in terms of Gaussian generators, we have
\beq\label{d12app}
s_{11}^{(1)}=d_1^{(1)}=0,\quad s_{11}^{(2)}=d_1^{(2)}=2D_1^{(2)}-(D_1^{(1)})^2-(F^{(1)})^2.
\eeq
Similarly, we have
\be
\begin{split}
&s_{22}^{(1)}=d_2^{(1)}=0,\\
&s_{22}^{(2)}=d_2^{(2)}+f^{(1)}e^{(1)}=-(E^{(1)})^2+2D_2^{(2)}-(D_2^{(1)})^2+2F^{(1)}E^{(1)}.
\end{split}
\ee
Note that by \eqref{sym u}, we have $e^{(1)}=s_{12}^{(1)}=-s_{21}^{(1)}=-f^{(1)}$. It follows from \eqref{b10app} and the above equation that
\beq\label{d22app}
d_2^{(2)}=2D_2^{(2)}-(D_2^{(1)})^2-[E^{(1)},F^{(1)}]+(F^{(1)})^2.
\eeq
Because $d_1^{(1)}=d_2^{(1)}=0$, we find from \eqref{eq:Y2i-R} that
\beq\label{app3}
\begin{split}
h_{1,1}&=d_2^{(2)}-d_1^{(2)}\\
&\overset{\eqref{d12app}}{\underset{\eqref{d22app}}{=}}2D_2^{(2)}-(D_2^{(1)})^2-2D_1^{(2)}+(D_1^{(1)})^2+2(F^{(1)})^2-[E^{(1)},F^{(1)}].    
\end{split}
\eeq
% Similarly, by \eqref{eq:Y2-R}, we have
% \beq\label{app2}
% \begin{split}
% &\xi_{1,0}=D_2^{(1)}-D_1^{(1)},\\
% &\xi_{1,1}=D_2^{(2)}-D_1^{(2)}+(D_1^{(1)})^2+\hf D_2^{(1)}-\hf D_1^{(1)}-D_1^{(1)}D_2^{(1)}.
% \end{split}
% \eeq
Recall the following obvious equalities $[E^{(1)},F^{(1)}]=D_1^{(1)}-D_2^{(1)}$ and $[D_1^{(1)},D_2^{(1)}]=0$ from \eqref{RTT} (or \cite{BK05}). Combining \eqref{app1} and \eqref{app3}, we obtain the second equality.
\end{proof}

\subsection{Type $\mathsf C_2$}\label{sec:app2}
In this section, we establish Part (1) of Theorem \ref{thm:embedmaintext} for type $\mathsf C_2$ using the minimalistic presentation from Theorem \ref{thm:min-text} (with the extra relation \eqref{eq:add-rel}). Since all other relations have been verified in \S \ref{sec:pfthmB0}, it suffices to verify the extra relation \eqref{eq:add-rel} for $i=1$.

Our strategy is as follows. We first construct elements $h_{1,2r+1},b_{1,r}$ for $r\in\bN$ from the twisted Yangian in R-matrix presentation \cite{GR16} and then prove that these elements satisfy the relations \eqref{ty0}--\eqref{ty2} (with $i=j=1$). Finally, we show by a similar calculation as in \S \ref{sec:app} that as elements in $\Y$, $h_{1,1}$ and $b_{1,0}$ are given by  the RHS of \eqref{hi1-embedding}--\eqref{bi0-embedding}, respectively. Then the extra relation \eqref{eq:add-rel} is satisfied due to the same computation in Lemma \ref{lem:hbhb}. 

We only sketch the proof for the first step. Note that we only treat with the first node and hence we do not need to do a rank reduction or a full study of the Gauss decomposition.

Recall the R-matrix presentation of the twisted Yangian of type $\mathsf{CI}_2$ (split type $\mathsf C_2$) from \cite{GR16}, and we adopt the same notation as therein. As in \cite{LWZ25GD,LZ24}, we need to carefully pick a matrix $\mathcal G$. The choice in \cite{GR16} is given by
\[
\mathcal G=\begin{bmatrix}
-1 & 0 & 0 & 0\\
0 & -1 & 0 & 0\\
0 & 0 & 1 & 0\\
0 & 0 & 0 & 1
\end{bmatrix}.
\]
Our choice of the matrix $\mathcal G'$ for Gauss decomposition is 
\[
\mathcal G'=A^t\mathcal G A=\begin{bmatrix}
0 & 0 & 0 & 1\\
0 & 0 & 1 & 0\\
0 & 1 & 0 & 0\\
1 & 0 & 0 & 0
\end{bmatrix},\qquad A=\frac{\sqrt 2}{2}\begin{bmatrix}
1 & 0 & 0 & 1\\
0 & 1 & 1 & 0\\
0 & -1 & 1 & 0\\
-1 & 0 & 0 & 1
\end{bmatrix}.
\]
Note that $t$ stands for the modified transpose defined in \cite[\S2]{GR16}. By \cite[Remark 3.2]{GR16}, these two choices of the matrix $\mathcal G$ give rise to isomorphic twisted Yangians.

Let $t_{ij}(u)$ (resp. $\mathsf s_{ij}(u)$) for $i,j=-2,-1,1,2$ be the generating series of the R-matrix presentation of Yangian (resp. twisted Yangians) of type $\mathsf C_2$ (resp. $\mathsf{CI}_2$), see e.g. \cite[\S2--\S4]{GR16}. Again as in type $\mathsf A$, we use $S(u)=T^t(-u)\mathcal G'T(u)$. Then with the choice of matrix $\mathcal G'$, we have
\[
\mathsf s_{ij}(u+\hf\kappa)=\sum_{a=-2}^{2}\mathrm{sign}(a)\mathrm{sign}(i)t_{-a,-i}(-u)t_{-a,j}(u).
\]
Then we introduce a new matrix $\mathscr{L}(u)=(l_{ij}(u))$, where $l_{ij}(u)=\sfs_{-i,j}(u+\hf\kappa)$ for $i,j=-2,-1,1,2$ (following a similar strategy to \cite{LZ24} for quasi-split type $\mathsf A$). Then it follows from \cite[(4.4)--(4.5)]{GR16} that for $i,j<0$, $l_{ij}(u)$ satisfy the relations \eqref{quater u}--\eqref{sym u} (as all these $\delta$'s in \cite[(4.4)]{GR16} vanish). Hence the submatrix $\mathscr{L}^-(u)=(l_{ij}(u))_{i,j=-2,-1}$ is the $S$-matrix for twisted Yangian of type $\mathsf{AI}$ (rank 1).  Define the generating series $h_1(u),b_1(u)$ via the Gauss decomposition for the submatrix $\mathscr{L}^-(u)$ as in \cite{LWZ25GD} (see also Appendix \ref{sec:app}). Since the generating series $l_{ij}(u)$ for $i,j=-2,-1$ satisfy the same relations as $s_{ij}(u)$ for $i,j=1,2$, we find that the coefficients $h_{1,2r+1},b_{1,r}$ for $r\in\bN$ also satisfy the relations \eqref{ty0}--\eqref{ty2} (with $i=j=1$), see \cite[Proposition 4.3]{LWZ25GD}.

The precise relations between $t_{ab}(u)$ and $\xi_i(u),x_i^\pm(u)$ are described in \cite{JLM18} which are very similar to type $\mathsf A$ in Appendix \ref{sec:app}. The rest of the calculation is very similar to that in Appendix \ref{sec:app}, hence we do not repeat it.

\bibliographystyle{amsalpha}
\bibliography{reference}

\end{document}